\documentclass[12pt]{amsart}
\usepackage[margin=2.5cm]{geometry}
\usepackage[svgnames]{xcolor}
\usepackage{enumitem}
\usepackage{amsmath,amssymb}
\usepackage{caption}

\usepackage{array}   
\newcolumntype{C}{>{$}c<{$}} 
\usepackage{arydshln} 

\usepackage{tikz}
\usepackage[latin1]{inputenc}
\usepackage[english]{babel}
\usepackage{multimedia}
\usepackage{multicol}
\usepackage[T1]{fontenc} 
\usepackage{times}
\usepackage{ytableau}
\usepackage{graphicx}
\usepackage{fourier}
\DeclareMathOperator{\CT}{ct}
\newcommand{\jack}[2]{f_{#1,#2}}

\newcommand{\wti}[3]{\textrm{wt}_{#1,#2}(#3)}
\newcommand{\wt}[2]{\textrm{wt}(#1,#2)}
\newcommand{\ct}[2]{\textrm{ct}({#1,#2})}
\newcommand{\ctBox}[1]{\textrm{ct}({#1})}

\newcommand{\omitt}[1]{}

\newcommand{\mc}[1]{\mathcal{#1}}

\newcommand{\CC}{\mathbf{C}}
\newcommand{\QQ}{\mathbf{Q}}
\newcommand{\RR}{\mathbf{R}}
\newcommand{\ZZ}{\mathbf{Z}}
\newcommand{\OO}{\mathcal{O}}
\newcommand{\ttt}{\mathfrak{t}}

\newcommand{\Z}{\mathbf Z}
\newcommand{\Pnk}{{\mathcal P}(n,k)}
\newcommand{\R}{\mathcal R}
\newcommand{\I}{\mathcal I}

\newcommand{\bgspart}[2]{b(#1,#2)}
\newcommand{\bgsset}{P_k(\bgspart{n}{k})}
\newcommand{\unders}[1]{{\mathcal U}(#1)}
\newcommand{\M}[2]{{\mathcal M}(#1,#2)}
\newcommand{\mult}[3]{{\text{mult}(#1,#2,#3)}}
\newcommand{\multB}[2]{{\text{mult}(#1,#2)}}
\newcommand{\rev}[1]{#1^{\textrm{REV}}}
\newcommand{\phiInv}[2]{\phi^{-1}(#1,#2)}
\newcommand{\la}{\langle}
\newcommand{\ra}{\rangle}
\def\Map{\phi}
\def\iop{\tau}
\newcommand{\Mapa}[1]{\Map(#1)}
\newcommand{\mapconst}[1]{b_{#1}}
\newcommand{\sigconst}[3]{K(#1,#2,#3)}

\def\zeroComp{\mathbf{0}}
\newcommand{\jacko}[1]{\jack{\zeroComp}{#1}}
\def\reg{\rho}
\def\col{c_0}
\def\cola{c_1}

\def\standComp{M}
\def\gMinusl{\gamma\setminus\lambda}
\def\lMinusg{\lambda\setminus\gamma}
\def\Sym{\mathfrak S}
\def\TIm{T'}
\def\TDo{T}
\def\SIm{S'}
\def\SDo{S}

\usepackage{chngcntr}
\counterwithin{figure}{section}
\DeclareMathOperator{\syt}{SYT}

\DeclareMathOperator{\inv}{Inv}

\numberwithin{equation}{section}
\theoremstyle{definition}
  \newtheorem{theorem}{Theorem}[section]
  
  \newtheorem{proposition}[theorem]{Proposition}
  \newtheorem{lemma}[theorem]{Lemma}
  \newtheorem{claim}[theorem]{Claim}
  \newtheorem {corollary}[theorem]{Corollary}
  \newtheorem{definition}[theorem]{Definition}
  \newtheorem{example}[theorem]{Example}
  \theoremstyle{remark}
    \newtheorem{notation}[theorem]{Notation}

\newcommand{\stanmod}{\Delta_c}
\usepackage[colorinlistoftodos]{todonotes}

\newcommand{\Susanna}[1]{\todo[size=\tiny,inline,color=red!30]{#1
      \\ \hfill --- Susanna}}
\newcommand{\sftodo}[1]{\Susanna{#1}}

\newsavebox{\smlmat}
\savebox{\smlmat}{$\left(\begin{smallmatrix}1&2&3&4&5&6\\2&4&5&3&6&1\end{smallmatrix}\right)$}

\newsavebox{\smlmatInv}
\savebox{\smlmatInv}{$\left(\begin{smallmatrix}1&2&3&4&5&6\\6&1&4&2&3&5\end{smallmatrix}\right)$}

\ytableausetup
{mathmode, boxsize=1.2em}

\title{Unitary representations of the Cherednik algebra: $V^*$-homology}

\author{Susanna Fishel}

\author{Stephen  Griffeth}

\author{Elizabeth Manosalva}

\begin{document}

\begin{abstract}
We give a non-negative combinatorial formula, in terms of Littlewood-Richardson numbers, for the $V^*$-homology of the unitary representations of the cyclotomic rational Cherednik algebra, and as a consequence, for the graded Betti numbers for the ideals of a class of subspace arrangements arising from the reflection arrangements of complex reflection groups. \end{abstract}

\thanks{The second author acknowledges the financial support of Fondecyt Proyecto Regular 1190597.}

\maketitle

\section{Introduction}

\subsection{}  

This article is a contribution to the study of the unitary representations of the rational Cherednik algebra and their interaction with the combinatorial commutative algebra around arrangements of linear subspaces arising from complex reflection groups of the type considered in \cite{LiLi} and \cite{Sid}. The study of the unitary representations of rational Cherednik algebras was initiated by Etingof and Stoica in \cite{EtSt}, motivated by a question posed by Cherednik. Together with the second author, Etingof and Stoica obtained the classification of the irreducible unitary representations in category $\OO_c$ of the type $S_n$ rational Cherednik algebra; subsequently, the second author obtained the classification and graded dimensions of the irreducible unitary representations in category $\OO_c$ of the type $G(\ell,1,n)$ rational Cherednik algebra \cite{Gri3} (which we will refer to as the \emph{cyclotomic rational Cherednik algebra}). In this paper we build on these works together with results of Ciubotaru \cite{Ciu} and Huang-Wong \cite{HuWo} to refine the calculation of the graded dimension to obtain a manifestly non-negative Kazhdan-Lusztig type character formula, in terms of Littlewood-Richardson numbers.  

On the other hand,  the paper \cite{Gri4} studies the class of subspace arrangements to which Cherednik algebra methods may be applied, and in particular describes when the socle of the polynomial representation of the cyclotomic Cherednik algebra is a unitary representation and the ideal of a subspace arrangement. By combining that paper with this one, we obtain a combinatorial formula for the graded equivariant Betti numbers of a class of linear subspace arrangements containing the class studied in \cite{Sid} and generalizing the $k$-equals arrangement studied in \cite{LiLi}. In this direction, the paper \cite{BNS} proves the existence of a certain minimal resolution of the ideal of the $k$-equals arrangement (as conjectured in \cite{BGS}); here we give an explicit description, in terms of the representation-valued orthogonal polynomials introduced in \cite{Gri2}, of the maps between the standard modules appearing in this minimal resolution. 

Each object of category $\OO_c$ is in particular a graded $W$-representation. So given an irreducible object $L$ of $\OO_c$, its graded character $\mathrm{ch}(L)$ may be thought of as a formal power series in one variable whose coefficients are characters of $W$. On the other hand, category $\OO_c$ is a highest weight category with standard $\Delta_c(E)$ and irreducible $L_c(E)$ objects indexed by the irreducible complex representations $E$ of $W$. Knowledge of the groups $\mathrm{Ext}^i(\Delta_c(F),L_c(E))$ for all $i$ and $F$ is then a second sort of graded character formula for $L_c(E)$ (analogous to the Kazhdan-Lusztig character in Lie theory). These Ext groups may be conveniently packaged together in the $V$-cohomology groups $H^i(V,L_c(E))$  (or, as observed in \cite{Gri4}, the $V^*$-homology groups $H_i(V^*,L_c(E))$), which are analogous to the $\mathfrak{n}$-cohomology groups in Lie theory. 

The first main theorem of this article contains a combinatorial description of the $V^*$-homology groups for all unitary representations $L_c(E)$ in category $\OO_c$ of the cyclotomic rational Cherednik algebra, and a combinatorial description of the graded characters of the somewhat wider class of $\ttt$-diagonalizable modules, where $\ttt$ is a certain commutative subalgebra of the cyclotomic rational Cherednik algebra (unlike the unitary representations, which have a tendency to be infinite dimensional, there are very many interesting finite dimensional $\ttt$-diagonalizable representations). The second main theorem gives an explicit description, in terms of the analogs of Jack polynomials introduced by the second author in \cite{Gri2}, of the maps between the standard modules appearing in the minimal resolution of the $k$-equals ideal.

\subsection{Statement of results} Here we will state our main theorems (postponing to \ref{combinatorics} a detailed explanation of the combinatorics involved in the statement of Theorem \ref{main}), followed by a more explicit version for $\ell=1$ and $\ell=2$ in \ref{comb notation} and an application of the result to subspace arrangements and a small worked example in \ref{example}. For the moment, to orient the reader we mention: an $\ell$-partition $\lambda$ of $n$ is simply a sequence $\lambda=(\lambda^0,\lambda^1,\dots,\lambda^{\ell-1})$ of $\ell$ partitions with $n$ total boxes. The $\ell$-partitions of $n$ index the irreducible complex representations of $G(\ell,1,n)$, and we write $P_{\ell,n}$ for the set of all $\ell$-partitions of $n$. We write $\mathrm{ct}_c(\lambda)$ for the sum of the charged contents of the boxes of $\lambda$ (this is essentially the $c$-function from \ref{c function} of the representation indexed by $\lambda$), and, hoping it will not cause confusion, use the standard notation $c^D_\mu$ and $c^\lambda_{\mu \nu}=c^{\lambda \setminus \mu}_\nu$ for a cyclotomic version of the Littlewood-Richardson numbers (where $D$ is a skew $\ell$-diagram and $\lambda$, $\mu$, and $\nu$ are $\ell$-partitions). 

In addition to these ingredients, which will be familiar to experts, there are two more specialized combinatorial constructions required for the statement of our main theorem. Firstly, we will define a certain set $\mathrm{Tab}_c(\lambda)$ of tableaux on each $\lambda \in P_{\ell,n}$, which depends on the deformation parameter $c$, and we write $|Q|$ for the sum of the entries of $Q \in \mathrm{Tab}_c(\lambda)$.  Secondly, for each element $Q \in \mathrm{Tab}_c(\lambda)$ we will define a certain skew diagram $s_c(Q)$, whose construction likewise. depends in an important way on the deformation parameter. These ingredients given, we may now state the theorem.

\begin{theorem} \label{main} Let $\lambda$ be an $\ell$-partition of $n$.
\item[(1)] If $L_c(\lambda)$ is $\ttt$-diagonalizable, then
$$\mathrm{ch}(L_c(\lambda))=\sum_{\substack{Q \in \mathrm{Tab}_c(\lambda) \\ \mu \in P_{\ell,n}}} c^{s_c(Q)}_\mu [S^\mu] t^{|Q|}.$$ \item[(2)] If $L_c(\lambda)$ is unitary, then for each $\ell$-partition $\mu$ of $n$
$$
\mathrm{dim}_\CC\left[ \mathrm{Ext}^i_{\OO_c}(\Delta_c(\mu),L_c(\lambda))\right]=\sum_{\substack{Q \in \mathrm{Tab}_c(\lambda,)\\ \nu \in P_{\ell,n}, \ \eta \in P_{\ell,n-i}, \ \chi \in P_{\ell,i} \\ |Q|=\mathrm{ct}_c(\lambda)-\mathrm{ct}_c(\mu)-i}} c^{s_c(Q)}_\nu c^\nu_{\eta \chi} c^\mu_{\eta \chi^t},
$$
\end{theorem}

We note that the connection between Littlewood-Richardson numbers and graded multiplicities arising in representation theory has been observed before, for instance in \cite{ChTa}, \cite{LeTh} and \cite{LeMi}.

In analogy to formulas for $\mathfrak{n}$-cohomology of irreducible highest-weight representations of Lie algebras, one does not hope for explicit formulas for arbitrary irreducibles in category $\OO_c$: in general, the algorithm provided by Kahzdan-Lusztig polynomials (see e.g. \cite{RSVV}) is the best we can expect. But for the finite-dimensional $\ttt$-diagonalizable representations and the unitary representations, the theorem gives an effective (and practical) algorithm for computing (two different kinds of) characters. In \ref{comb notation} we will make the combinatorics more explicit still for $\ell=1$ and $\ell=2$. One interesting feature of these formulas is the following: for small values of $\ell$ and $n$, it seems that the dimensions of the relevant Ext groups are always $0$ or $1$. On the other hand, the form of the formula makes this seem a bit surprising. It would be interesting to collect more extensive numerical data. 

In \ref{BGG} we prove a second sort of result: we construct explicit polynomial maps between modules in the BGG resolution of a unitary representation of the type $A$ Cherednik algebra. This resolution was proved to exist in \cite{BNS}, confirming a conjecture of \cite{BGS}, by using ideas of Webster \cite{Web} to translate the problem into the world of diagram algebras. This raises a natural question: do unitary representations in general possess resolutions of BGG type by standard objects? At the moment we do not know of counterexamples, but neither do we have structural reasons for believing that BGG resolutions should exist, outside of the naive expectation that experience with classical Lie theory suggests. 

We note that both the existence and uniqueness statements of the next theorem are previously known (existence was proved in \cite{}, and uniqueness is known at least since \cite{BNS}). We do give an independent proof of these facts, but the main point here is the explicit polynomial formula, in terms of the representation-valued orthogonal functions $f_{\nu,S}$ from \cite{Gri3}, which are indexed by a pair consisting of a composition $\nu$ and a standard Young tableau $S$. 

\begin{theorem} \label{main 2}
  In case $\ell=1$, suppose $c=1/k$, and let $\lambda$ and $\gamma$ be
  partitions in the poset $\Pnk$ (see \ref{subsec:poset})
  such that $\gamma$ covers $\lambda$. Then there is a unique (up to
  scalars) homomorphism $\Delta_c(\gamma) \to \Delta_c(\lambda)$ of
  $H_c(S_n,\CC^n)$-modules, which may be chosen such that $f_{\zeroComp,T_0}
  \mapsto f_{\mu,S}$, where $(\mu,S)=\phi(\zeroComp,T_0)$ is as described in
  \ref{subsec:maps} and $T_0$ is the row-reading tableau on $\gamma$.
\end{theorem}

To obtain our results we use four main tools: firstly, the Hodge decomposition theorem for Dirac cohomology of rational Cherednik algebras as introduced in \cite{Ciu} and \cite{HuWo}; secondly, the Specht-module valued orthogonal functions generalizing Jack polynomials introduced in \cite{Gri2} as applied in \cite{Gri3}; thirdly, the representation theory of the cyclotomic degenerate affine Hecke algebra (this is quite similar to the case $\ell=1$ and we develop what we need of it in \ref{cyclotomic}); and finally, the combinatorics of partitions, including abaci, the Littlewood-Richardson rule and jeu de taquin. These combinatorics enter in three types of calculation: for the restriction rule from the cyclotomic degenerate affine Hecke algebra to the group algebra of the symmetric group, for proscribing poles of orthogonal functions, and for the description of the tensor product of $G(\ell,1,n)$-modules by the exterior powers of the dual of its reflection representation. 

\subsection{Combinatorics}  \label{combinatorics}

A \emph{partition} is a non-increasing sequence $\lambda_1 \geq \lambda_2 \geq \cdots \geq \lambda_\ell >0$ of positive integers. As usual we will identify partitions with their Young diagrams and visualize them as collections of boxes. The \emph{diagram} of $\lambda$ is the collection of points $(x,y)$ for $x$ and $y$ integers satisfying $1 \leq y \leq \ell$ and $1 \leq x \leq \lambda_y$, which we think of as a subset of $\RR^2$. More generally, we call a finite subset $D \subseteq \RR^2$ a \emph{skew shape} if whenever $(x,y) \in D$ and $(x+k,y+\ell) \in D$ for non-negative integers $k$ and $\ell$, then $(x+k',y+\ell') \in D$ for all integers $0 \leq \ell' \leq \ell$ and $0 \leq k' \leq k$. We will often think of the points of a skew shape as boxes. We define the \emph{content} of a box $b=(x,y)$ by 
$$\mathrm{ct}(b)=x-y.$$

A skew shape $D$ is \emph{integral} if $D \subseteq \ZZ^2_{>0}$. Each integral skew shape is the difference $D=\alpha \setminus \beta$ of two partitions (which are not uniquely determined by $D$). Boxes $b=(x,y)$ and $b'=(x',y')$ are \emph{adjacent} if either 
\begin{enumerate}
\item[(a)] $x=x'$ and $y=y' \pm 1$, or
\item[(b)] $y=y'$ and $x=x' \pm 1$. 
\end{enumerate} The equivalence classes for the equivalence relation on the boxes of a skew shape $D$ generated by adjacency are called the \emph{connected components} of $D$, and $D$ is \emph{connected} if it has only one connected component. 

Fixing an integer $\ell>0$, an $\ell$-partition is a sequence $\lambda=(\lambda^0,\lambda^1,\dots,\lambda^{\ell-1})$ of partitions (some of which may be empty). Likewise, an $\ell$-skew shapes is a sequence $D=(D^0,D^1,\dots,D^{\ell-1})$ of $\ell$ skew shapes, some of which may be empty. A box of $D$ is a box of some $D^j$, and if $b \in D^j$ is a box of $D^j$ we write 
$$\beta(b)=j.$$ If there are $n$ total boxes in $\lambda$ then we say that $\lambda$ is an \emph{$\ell$-partition of $n$}, and we write $P_{\ell,n}$ for the set of all $\ell$-partitions of $n$. Given $\lambda \in P_{\ell,n}$ we define its \emph{transpose} $\lambda^t$ by
$$\lambda^t=((\lambda^{1})^t,(\lambda^2)^t,\dots,(\lambda^{\ell-1})^t,(\lambda^0)^t),$$ where $\mu^t$ is the transpose of the partition $\mu$. (With our indexing of the complex irreducible representations of $G(\ell,1,n)$, the representation indexed by $\lambda^t$ is the tensor product of the representation $\Lambda^n V^*$, where $V=\CC^n$ is the defining representation, with the representation indexed by $\lambda$).

We order the boxes of each $\ell$-skew shape $D$ (and in particular, each $\ell$-partition) as follows:
$$b \leq b' \ \iff \ \beta(b)=\beta(b') \ \text{and} \ b=(x,y), \ b'=(x',y') \quad \text{with} \quad x' -x \in \ZZ_{\geq 0} \ \text{and} \ y'-y \in \ZZ_{\geq 0}.$$

Fixing positive integers $\ell$ and $n$, we write $G(\ell,1,n)$ for the group of  $n$ by $n$ matrices with exactly one non-zero entry in each row and each column, and such that the non-zero entries are $\ell$th roots of $1$. The irreducible complex representations of $G(\ell,1,n)$ are in bijection with $\ell$-partitions $\lambda$ of $n$, and for $\lambda \in P_{\ell,n}$ an $\ell$-partition of $n$ we write $S^\lambda$ for the corresponding irreducible representation. We write $V=\CC^n$ for the defining representation of $G(\ell,1,n)$ (which is irreducible if $\ell>1$). When $\ell$ is fixed we may abbreviate $G_n=G(\ell,1,n)$. 

We write $H_c$ for the rational Cherednik algebra attached to this group, $\ttt \subseteq H_c$ for the Dunkl-Opdam subalgebra, $\OO_c$ for the usual subcategory of $H_c$-mod, and  $L_c(\lambda)$ for the irreducible object of $\OO_c$ corresponding to an $r$-partition $\lambda$ of $n$ (see Section \ref{Prelim} for precise definitions). The subscript $c$ stands for the central parameters, which are a list 
$$c=(c_0,d_0,\dots,d_{\ell-1}) \in \RR^{\ell+1}$$ with $d_0+\cdots+d_{\ell-1}=0$. These parameters give rise to statistics on $\ell$-partitions as follows: given an $\ell$-partition $\lambda$ and a box $b \in \lambda$, we define its \emph{charged content} $\mathrm{ct}_c(b)$ by
$$\mathrm{ct}_c(b)=d_{\beta(b)}+\ell\mathrm{ct}(b) c_0,$$ and we define the \emph{charged content} of $\lambda$ to be the sum of the charged contents of its boxes
$$\mathrm{ct}_c(\lambda)=\sum_{b \in \lambda} \mathrm{ct}_c(b).$$ This statistic is essentially the $c$-function (see \ref{c function}) of the representation indexed by $\lambda$. 

\label{sec:stanModBasis}
For generic parameters $c$, each standard module has a basis $f_{P,Q}$ consisting of eigenfunctions for the Dunkl-Opdam subalgebra $\ttt$. For certain special values of the parameters these may have poles. However, each irreducible quotient $L_c(\lambda)$ that is $\ttt$-diagonalizable has a certain subset of these $f_{P,Q}$ as a basis; this subset is indexed by pairs $(P,Q)$ of tableaux on $\lambda$ with the following properties: 
\begin{itemize}
\item[(a)] $Q$ is a filling of the boxes of $\lambda$ by non-negative integers such that $Q(b) \leq Q(b')$ whenever $b \leq b'$,
\item[(b)] $P$ is a bijection from the boxes of $\lambda$ to the set of integers $\{1,2,\dots,n\}$, such that if $b \leq b'$ and $Q(b)=Q(b')$ then $P(b) > P(b')$, 
\item[(c)] If $b$ is a box of $\lambda$ and $k$ is a positive integer such that
$$\mathrm{ct}_c(b)=d_{\beta(b)-k}+k$$ then $Q(b) < k$, and
\item[(d)] If $b$ and $b'$ are boxes of $\lambda$ and $k$ is a positive integer with $k=\beta(b)-\beta(b')$ mod $\ell$ and such that 
$$\mathrm{ct}_c(b)-\mathrm{ct}_c(b')=k \pm \ell c_0$$ then $$Q(b) \leq Q(b')+k \quad \text{with equality implying} \quad P(b) > P(b').$$
\end{itemize} We write $\Gamma_c(\lambda)$ for the set of such pairs $(P,Q)$, and define $\mathrm{Tab}_c(\lambda)$ to be the set of $Q$ such that there exists a $P$ with $(P,Q) \in \Gamma_c(\lambda)$, so that $\mathrm{Tab}_c(\lambda)$ is the projection of $\Gamma_c(\lambda)$ on its second coordinate
$$\mathrm{Tab}_c(\lambda)=\pi_2(\Gamma_c(\lambda)).$$ 

For each $Q \in \mathrm{Tab}_c(\lambda)$ there is a (unique up to diagonal slides of its connected components) skew diagram $s_c(Q)$ defined as follows: choose any $P$ with $(P,Q) \in \Gamma_c(\lambda)$. Theorem \ref{diag class} and the first line of the proof of Theorem \ref{main version} give an algorithm determining a standard Young tableau $T$ with $$\beta(T^{-1}(i))=\beta(P^{-1}(n-i+1))-Q(P^{-1}(n-i+1))$$ and $$\mathrm{ct}(T^{-1}(i))=\mathrm{ct}(P^{-1}(n-i+1))-\frac{1}{\ell c_0}(Q(P^{-1}(n-i+1))-(d_{\beta(P^{-1}(n-i+1))}-d_{\beta(P^{-1}(n-i+1))-Q(P^{-1}(n-i+1))})).$$ Then we define $s_c(Q)$ to be the shape of $T$, which is independent of the choice of $P$. We define the \emph{degree} $|Q|$ of $Q$ by
$$|Q|=\sum_{b \in \lambda} Q(b).$$ 

Given $\ell$-partitions $\alpha \vdash m+n$, $\beta \vdash m$, and $\gamma \vdash n$, the (cyclotomic) \emph{Littlewood-Richardson number} $c^\alpha_{\beta \gamma}$ is the dimension
$$c^\alpha_{\beta \gamma}=\mathrm{dim}_\CC(\mathrm{Hom}_{G_{m+n}}(\mathrm{Ind}^{G_{m+n}}_{G_m \times G_n}(S^\beta \otimes S^\gamma),S^\alpha)).$$ By Frobenius reciprocity and the tensor-hom adjunction we also have
$$c^\alpha_{\beta \gamma}=\mathrm{dim}_\CC(\mathrm{Hom}_{G_{n}}(S^\gamma,\mathrm{Hom}_{G_m}(S^\beta,S^\alpha))).$$ 

Given an integral skew shape $D=\alpha \setminus \beta$ we define
$$c^D_\gamma=c^{\alpha \setminus \beta}_\gamma=c^\alpha_{\beta \gamma},$$ which is independent of the choice of $\alpha$ and $\beta$, and somewhat more generally, if $D$ is an arbitrary skew shape then for any integral skew shape $D'$ whose connected components may be obtained from those of $D$ by horizontal slides, we define 
$$c^D_\mu=c^{D'}_\mu.$$ This definition is independent of the choice of such a $D'$. In this fashion we have defined Littlewood-Richardson numbers $c^D_\mu$ for all $\ell$-skew shapes $D$, and which, as is explained in \ref{branching}, controls the branching rule from the irreducible affine Hecke algebra module indexed by $D$ to $G_n$. 

These cyclotomic Littlewood-Richardson numbers are related to the classic ($\ell=1$) Littlewood-Richardson numbers in a very simple way: if $D=(D^0,D^1,\dots,D^{\ell-1})$ and $\mu=(\mu^0,\mu^1,\dots,\mu^{\ell-1})$ then
$$c^D_\mu=\prod_{j=0}^{\ell-1} c^{D_j}_{\mu_j}.$$ Moreover, these classical LR numbers may be computed via a number of combinatorial constructions. Here we describe the most classical: A \emph{Littlewood-Richardson tableau} on a skew diagram $D$ (for $\ell=1$) is a function $T: D \rightarrow \ZZ_{>0}$ such that
\begin{itemize}
\item[(a)] The tableau $T$ is \emph{column strict} in the sense that $T(x,y)<T(x,y+1)$ whenever $(x,y), (x,y+1) \in D$ and $T(x,y) \leq T(x+1,y)$ whenever $(x,y),(x+1,y) \in D$, and
\item[(b)] The row-reading word $T_1 T_2 \cdots T_n$ (obtained by reading the entries of $T$ from top to bottom and right to left) of $T$ satisfies the \emph{LR property}: for each integer $i \in \ZZ_{>0}$ and each $1 \leq k \leq n$, the number of occurrences of $i$ in the sequence $T_1 T_2 \cdots T_k$ is at least as large as the number of occurrences of $i+1$.  
\end{itemize} The \emph{weight} of a tableau $T$ is the sequence $\nu_1,\nu_2,\dots$ where $\nu_i$ is the number of boxes $b\in D$ with $T(b)=i$. The Littlewood-Richardson coefficient $c^D_\nu$ is then the number of Littlewood-Richardson tableaux on $D$ of weight $\nu$.

\subsection{Combinatorics in type A and B} \label{comb notation} Here we make part (2) of Theorem \ref{main} more explicit for $\ell=1$ and $\ell=2$, corresponding to the Weyl groups of types $A$ and $B$. We begin with the simpler case $\ell=1$, and then explain the modifications necessary for $\ell=2$. We will explain in both cases how to interpret everything in the sum on the right-hand side of part (2) of Theorem \ref{main} for these cases. (Also, though the theorem is stated in a uniform fashion, to apply it in examples one must use the classification theorems from \cite{EtSt} and \cite{Gri3}.) 

Firstly, in type $A$ the parameter $c$ is simply a number, and the statistic $\mathrm{ct}_c(b)$ is 
$$\mathrm{ct}_c(b)=c \cdot \sum_{b \in \lambda} \mathrm{ct}(b).$$ The most interesting unitary representations occur for $c=1/k$ with $k \geq 2$ an integer. For this choice of $c$, a partition $\lambda$ indexes a unitary representation if and only if the hook extending from the box of $\lambda$ of smallest content to the removable box of $\lambda$ of largest content is of length at most $k$. We assume we have fixed such a $\lambda$.

We define the set $\mathrm{Tab}_{1/k}(\lambda)$ of \emph{$k$-admissible tableaux on $\lambda$} as follows: a $k$-admissible tableaux on $\lambda$ is a filling $Q: \lambda \to \Z_{\geq 0}$ of the boxes of $\lambda$ by non-negative integers which is weakly increasing left to right and top to bottom (with the English conventions), and in addition satisfies the following admissibility condition that depends on $k$:
$$Q(b) \leq Q(b')+1 \quad \hbox{if $\mathrm{ct}(b)=\mathrm{ct}(b')+k-1$.}$$ 

Given a $k$-admissible tableau $Q$ on $\lambda$ we define a skew shape $s_{1/k}(Q)$ as follows: first, write $m=\lambda_1$ for the number of columns of $\lambda$. The skew shape $s(Q)$ is the union of the following skew shapes, indexed by non-negative integers $j$: for each non-negative integer $j$ form the skew shape consisting of the boxes $b$ of $\lambda$ with $Q(b)=j$, translated $j m$ positions to the left and $j(k-m)$ positions down in the plane. So, the boxes filled with $0$'s don't move, the boxes filled with $1$'s get translated $m$ positions to the left and $k-m$ positions down (decreasing their content by exactly $k$), the boxes filled with $2$'s get translated $2m$ positions to the left and $2(k-m)$ positions down (decreasing their contents by exactly $2k$) and so on. Thus for $k=4$ and $Q$ the filling 

$$ Q=\begin{ytableau}
0& 1 & 1 \\
0 & 1
\end{ytableau} $$ we have
$$
s_{1/4}(Q)=\ydiagram{2+1,3,1}
$$ 

These are the ingredients necessary to understand the formulae in Theorem \ref{main} for $\ell=1$. 

All the constructions for $\ell=1$ have $\ell=2$ analogs, which we now describe in order to explicitly compute using Theorem \ref{main} when $W=G(2,1,n)$ is the Weyl group of type $B$. In contrast to the case $\ell=1$, the classification of unitary representations is quite intricate, and we do not state it here (see \cite{Gri3}). 

The parameter is a pair $(c,d)$ of real numbers. We identify each $2$-partition $\lambda=(\lambda^0,\lambda^1)$ with its Young diagram, which is a subset of $\ZZ^2 \times \ZZ/2$. For the remainder of this subsection, we interpret \emph{skew shape} to mean a pair $D=(D^0,D^1)$ of skew shapes; we thus regard it as a subset of $\RR^2 \times \ZZ/2$. We fix a $2$-partition $\lambda$. The \emph{charged content} of a box $b$ of $\lambda$ is
$$\mathrm{ct}_c(b)=\begin{cases} d+ 2 \mathrm{ct}(b) c \quad \hbox{if $b \in \lambda^0$, and} \\ -d+2 \mathrm{ct}(b) c \quad \hbox{if $b \in \lambda^1$.} \end{cases} $$ and as above $\mathrm{ct}_c(\lambda)$ is the sum of the charged contents of the boxes of $\lambda$.

Now we define the set of \emph{$c$-admissible} tableaux $\mathrm{Tab}_c(\lambda)$. It consists of all tableaux $Q: \lambda \to \ZZ_{\geq 0}$ with the following properties:
\begin{itemize}
\item[(a)]  $Q(b) \leq Q(b')$ whenever $b \leq b'$, 
\item[(b)] $Q(b) < k$ if $k$ is an odd positive integer and $$d+\mathrm{ct}(b) c=k/2,$$ or if $k$ is an even positive integer and $$\mathrm{ct}(b) c=k/2,$$ and
\item[(c)] $Q(b) \leq Q(b')+k$ if $k=\beta(b)-\beta(b')$ mod $2$ and $\mathrm{ct}_c(b)-\mathrm{ct}_c(b')=k \pm \ell c$. 
\end{itemize}

We will now describe how to construct the skew shape $s_c(Q)$. In all cases it is obtained by applying the $k$th iterate $s_c^k$ of a certain function $s_c$ to each box $b \in \lambda$ with $Q(b)=k$. However, to define the function $s_c$, we must distinguish two cases. By symmetry we may assume $\lambda^0 \neq \emptyset$. 

Case 1. First assume either that $\lambda^1 \neq \emptyset$, or that $\lambda^1=\emptyset$ but the equation $$d+\mathrm{ct}(b) c =1/2$$ does not hold, where $b$ is the removable box of $\lambda^0$ of largest content. In this case, we define the \emph{$c$-shifting function} $s_c: \RR^2 \times \ZZ/2 \to \RR^2 \times \ZZ/2$ by
$$s_c(x,y,0)=(x-\lambda^0_1,y-\lambda^0_1+\frac{1}{2c}-\frac{d}{c},1) \quad \text{and} \quad s_c(x,y,1)=(x-\lambda^1_1,y-\lambda^1_1+\frac{1}{2c}+\frac{d}{c},0),$$ where if $\lambda^i=\emptyset$ then we interpret $\lambda^i_1=0$. 

Case 2. Now assume that $\lambda^1=\emptyset$ and the equation $$d+\mathrm{ct}(b) c =1/2$$ holds. In this case we put
$$s_c(x,y,0)=(x-p,y-p+\frac{1}{2c}-\frac{d}{c},1) \quad \text{and} \quad s_c(x,y,1)=(x,y+\frac{1}{2c}+\frac{d}{c},0),$$ where $p$ is the length of the second longest part of $\lambda^0$.

With this description of $s_c(Q)$ we have now specified all the ingredients needed to evaluate the sum in part (2) of Theorem \ref{main}. 

\subsection{Subspace arrangements and an example} \label{example} As observed in \cite{Gri4}, the socle of the polynomial representation of the type $G(\ell,1,n)$ rational Cherednik algebra has a tendency to be unitary when it is a radical ideal. This provides an application of the character formula proved here to computing the syzygies of certain ideals.

Fix integers $2 \leq k \leq n$ and $0 \leq m \leq k-1$ and define 
$$X_k=\{(x_1,\dots,x_n) \in \CC^n \ | \ \exists \ S \subseteq \{1,2,\dots,n\}, \ |S|=k, \ x_i^\ell=x_j^\ell \ \forall \ i,j \in S \}$$ and
$$Y_{m+1}=\{(x_1,\dots,x_n) \in \CC^n \ | \ \exists \ S \subseteq \{1,2,\dots,n\}, \ |S|=m+1, \ x_i=0 \ \forall \ i \in S \}.$$ Let $I$ be the ideal of the union $X_k \cup Y_{m+1}$. By Theorem 4.5 of \cite{Gri4}, for parameters satisfying $c_0=1/k$ and $d_0-d_{\ell-1}+\ell m c_0=1$ the ideal $I$ is an $H_c$-stable submodule of the polynomial representation, and provided the inequalities $d_0-d_{-p}+\ell m' c_0 < p$ hold for all pairs $(p,m')$ of integers satisfying the inequalities in part (d) of Theorem 1.6 of \cite{Gri4}, it is unitary (and thus equal to the socle; moreover the set of parameters for which this holds is non-empty). 

We describe its lowest weight, which is a certain $\ell$-partition $\lambda$. Combined with Theorem \ref{main}, this gives a combinatorial formula for the Betti numbers of $I \cong L_c(\lambda)$. Divide $n$ by $k-1$ to obtain a quotient $q$ and remainder $r$, with $n=q(k-1)+r$. We record the result of this division as the Young diagram of the partition $\lambda$ with $q$ parts of size $k-1$ and (if $r \neq 0$) one part of size $r$. We then split $\lambda$ into two partitions $\lambda^0$ and $\lambda^{\ell-1}$, by taking $\lambda^0$ to consist of the first $m$ columns of $\lambda$, and $\lambda^{\ell-1}$ to consist of the rest. Finally, put $\lambda=(\lambda^0,\emptyset, \dots,\emptyset,\lambda^{\ell-1})$. 

We now give a small example that illustrates all the preceding. Suppose $\ell=2$, $n=4$, $k=3$, and $m=1$. That is, we look at the ideal of the set of points in $\CC^4$ such that either some three coordinates are equal to one another up to a sign, or some two coordinates are zero. First of all, by the results of \cite{Gri4} it follows that this ideal is generated in degree $6$ by the $G(2,1,4)$-representation generated by the polynomial 
$$f(x)=x_1 x_2 (x_1^2-x_2^2)(x_3^2-x_4^2),$$ which is the irreducible representation indexed by
$$\lambda=\ydiagram{1,1} \ ,\ \ydiagram{1,1} . $$ Moreover, for $c=1/3$ and $d=1/6$, this is a unitary subrepresentation of the polynomial representation of the $G(2,1,4)$ Cherednik algebra. The charged content sum of $\lambda$ is $\mathrm{ct}_c(\lambda)=-4/3$, and the only possible $\mu \neq \lambda$ for which the sum on the right-hand side of part (2) of Theorem \ref{main} can be non-zero are therefore those which $\mathrm{ct}_c(\mu) \in \{-7/3,-10/3,-13/3,\dots \}$. There are two of these: 
$$\ydiagram{1} \ , \ \ydiagram{1,1,1} \quad \text{and} \quad  \ydiagram{1,1,1,1} \ , \ \emptyset $$ with charged content sums $-7/3$ and $-10/3$, respectively. 

For the first possible $\mu$ above, since $-4/3-(-7/3)=1$, we can have $\mathrm{Ext}^i(\Delta_c(\mu),L_c(\lambda)) \neq 0$ only for $i=1$, in which case $Q$ is the zero tableau and we have $\nu=\lambda$ in the sum. Now the Littlewood-Richardson number $c^\alpha_{\beta \gamma}$ is zero unless the diagram of $\beta$ is contained in that of $\alpha$, and so the product $c^\lambda_{\eta \chi} c^\mu_{\eta,\chi^t}$ can be non-zero only when the diagram of $\eta$ is contained in the intersection of $\lambda$ with $\mu$, which implies that 
$$\eta=\ydiagram{1} \ ,\ \ydiagram{1,1}  \quad \text{and } \quad \chi=\emptyset \ ,\ \ydiagram{1} $$ and the relevant Ext-dimension is just $1$. 

For the second possible $\mu$, since $-4/3-(-10/3)=2$, we might have a non-zero Ext in homological dimension $i=1$ or $i=2$.  If $i=2$ then again we must have $Q$ equal to the zero tableau, and $c^\lambda_{\eta \chi} c^\mu_{\eta,\chi^t}$ is zero for all choices of $\eta$ and $\chi$ since the intersection of $\lambda$ with $\mu^t$ is too small. Thus this second Ext-group is zero-dimensional. For $i=1$, there are two possible $Q$'s that contribute to the sum. Namely, $Q$ must be one of the following:
$$
Q=\begin{ytableau}
0 \\ 1
\end{ytableau} \ , \  \begin{ytableau}
0 \\ 0
\end{ytableau} \quad \text{or} \quad Q=\begin{ytableau}
0 \\ 0
\end{ytableau} \ , \ \begin{ytableau}
0 \\ 1
\end{ytableau}.$$ These produce the skew diagrams
$$s_c(Q)=\ydiagram{1} \ , \ \ydiagram{1+1,2} \quad \text{and} \quad s_c(Q)= \ydiagram{1+1,1+1,1} \ , \ydiagram{1}.$$ One checks as above that the first of these produces no contribution to the sum, and the second produces a contribution of $1$, with 
$$\nu=\ydiagram{1,1,1} \ , \ \ydiagram{1} \quad \eta=\ydiagram{1,1,1} \ , \ \emptyset \quad \text{and} \quad \chi=\emptyset \ , \ \ydiagram{1}.$$ Summing up, each of the $2$-partitions $\mu$ above has the relevant $Ext^1$ group of dimension $1$, and all other non-trivial Ext groups are zero. So this ideal is generated in degree $6$ by a six-dimensional space of generators, with a four-dimensional space of linear relations and a one-dimensional space of quadratic relations. There are no higher syzygies.

\section{Cherednik algebras, category $\OO$, and homology of unitary representations} \label{Prelim}

\subsection{The rational Cherednik algebra} Let $V$ be a finite dimensional $\CC$-vector space and let $W \subseteq \mathrm{GL}(V)$ be a finite group of linear transformations of $V$. The set of \emph{reflections} in $W$ is
$$R=\{ r \in W \ | \ \mathrm{codim}_V(\mathrm{fix}_V(r))=1 \}.$$ For each reflection $r \in R$ we fix $\alpha_r \in V^*$ and $c_r \in \CC$ with
$$\mathrm{fix}_V(r)=\{v \in V \ | \ \langle \alpha_r,v \rangle=0\} \quad \text{and} \quad c_{wrw^{-1}}=c_r \quad \hbox{for all $r \in R$ and $w \in W$.}$$ Each $y \in V$ defines a \emph{Dunkl operator} on $\CC[V]$, given by
$$y(f)=\partial_y(f)-\sum_{r \in R} c_r \langle \alpha_r,y \rangle \frac{f-r(f)}{\alpha_r} \quad \hbox{for $f \in \CC[V]$.}$$ The \emph{rational Cherednik algebra} is the subalgebra $H_c=H_c(W,V)$ of $\mathrm{End}_\CC(\CC[V])$ generated by $W$, the ring $\CC[V]$ acting on itself by multiplication, and the Dunkl operators for all $y \in V$.

\subsection{Real parameters} We call the parameter $c=(c_r)_{r \in R}$ \emph{real} if $c_{r^{-1}}=\overline{c_r}$ for all $r \in R$. We will assume for the entirety of this article that the parameter is real. This is no loss of generality, and ensures that the Fourier transform and contravariant forms on standard modules are well-defined (see below). 

\subsection{The PBW theorem} The rational Cherednik algebra may be presented as the algebra generated by the group algebra $\CC W$, the commuting operators $y \in V$, and the commuting operators $x \in V^*$, subject to the further relations
$$w x w^{-1}=w(x), \ w y w^{-1}=w(y), \ \text{and} \ yx-xy=\la x,y \ra -\sum_{r \in R} c_r \la \alpha_r,y \ra \la x,\alpha_r^\vee \ra r$$ for $w \in W$, $x \in V^*$, and $y \in V$, where $\alpha_r^\vee \in V$ is determined by the requirement
$$r(x)=x-\la x, \alpha_r^\vee \ra \alpha_r \quad \hbox{for all $x \in V^*$.}$$ The \emph{PBW theorem} states that multiplication induces an isomorphism from $\CC[V] \otimes \CC W \otimes \CC[V^*]$ onto the algebra with this presentation.

\subsection{Category $\OO_c$}  The category $\OO_c$ consists of all finitely-generated $H_c$-modules on which each Dunkl operator $y \in V$ acts locally nilpotently. 

\subsection{Standard modules} \label{sec:stanMod} Let $E \in \mathrm{Irr}(\CC W)$ be an irreducible representation. The \emph{standard module} $\Delta_c(E)$ for $H_c$ is the induced representation
$$\Delta_c(E)=\mathrm{Ind}_{\CC[V^*] \rtimes W}^{H_c}(E),$$ where we write $\CC[V^*] \rtimes W$ for the subalgebra of $H_c$ generated by $V$ and $W$, which acts on $E$ via
$$y e=0 \quad \hbox{for all $y \in V$ and $e \in E$.}$$ As a $\CC[V] \rtimes W$-module we have
$$\Delta_c(E) \cong \CC[V] \otimes E.$$

\subsection{The Fourier transform} We fix a $W$-invariant positive definite Hermitian form on $V$, inducing a $W$-equivariant conjugate linear isomorphism $V \rightarrow V^*$, written $y \mapsto \overline{y}$, with inverse $V^* \rightarrow V$ also written $x \mapsto \overline{x}$. If $W$ acts irreducibly on $V$ these are determined up to multiplication by a positive real number.

There is a conjugate linear  anti-automorphism of $H_c$, the \emph{Fourier transform}, that we write $h \mapsto \overline{h}$ for $h \in H_c$, induced by the isomorphisms $y \mapsto \overline{y}$ of $V$ onto $V^*$ and $x \mapsto \overline{x}$ of $V^*$ onto $V$ and the requirement $\overline{w}=w^{-1}$ for $w \in W$. Here we need the reality of the parameter $c$.

\subsection{The contravariant form} Fix a positive definite $W$-equivariant Hermitian form $(\cdot,\cdot)$ on $E$ (this is unique up to mulitplication by positive real numbers). The \emph{contravariant form} is the unique Hermitian form $(\cdot,\cdot)_c$ on $\Delta_c(E)$ that restricts to $(\cdot,\cdot)$ on $E$ and is such that
$$(hm_1,m_2)_c=(m_1,\overline{h} m_2)_c \quad \hbox{for all $m_1,m_2 \in \Delta_c(E)$ and $h \in H_c$.}$$ 

\subsection{Irreducible unitary representations} We write $L_c(E)$ for the quotient of $\Delta_c(E)$ by the radical of the contravariant form, and by abuse of notation write $(\cdot,\cdot)_c$ for the the induced form on $L_c(E)$. If this is positive definite, then we say that $L_c(E)$ is \emph{unitary}. In any case, $L_c(E)$ is an irreducible $H_c$-module, and as $E$ ranges over the irreducible complex representations of $W$, $L_c(E)$ ranges over the irreducible objects of $\OO_c$. 

\subsection{The Euler element} \label{c function} The \emph{Euler element} is the operator
$$\mathrm{eu}=\sum_{i=1}^n x_i \partial_{y_i},$$ where $y_1,\dots,y_n$ is a basis of $V$ with dual basis $x_1,\dots,x_n$ of $V^*$. It acts on a homogeneous polynomial $f$ of degree $m$ by 
$$\mathrm{eu}(f)=mf.$$ It happens that $\mathrm{eu} \in H_c$: this follows from the formula
$$\mathrm{eu}=\sum_{i=1}^n x_i y_i +\sum_{r \in R} c_r (1-r).$$

Given $a \in \CC$ and an $H_c$-module $M$, we define the $a$-weight space $M_a$ by
$$M_a=\{ m \in M \ | \ \hbox{there exists $N \in \ZZ_{>0}$ with $(\mathrm{eu}-a)^N m=0$} \}.$$ For the standard module $\Delta_c(E)$ we have
$$\Delta_c(E)=\bigoplus_{d \in \ZZ_{\geq 0}} \Delta_c(E)_{c_E+d},$$ where $c_E$ is the scalar by which $\sum_{r \in R} c_r(1-r)$ acts on $E$.

\subsection{Duality} Each $M \in \OO_c$ is the direct sum of its $\mathrm{eu}$-weight spaces,
$$M=\bigoplus_{a \in \CC} M_a,$$ or in other words each module in $\OO_c$ is graded by the eigenspaces for $\mathrm{eu}$. 

There is a duality $D$ on the category $\OO_c$ given by graded conjugate-linear duality of the underlying complex vector spaces. Specifically, for a $\CC$-vector space $U$ we write 
$$U^\vee=\{f:U \rightarrow \CC \ | \ f(u+v)=f(u)+f(v) \ \text{and} \ f(au)=\overline{a} f(u) \quad \hbox{for $u,v \in U$ and $a \in \CC$} \}$$ for the conjugate-linear dual of $U$. Given a representation $M$ in $\OO_c$ with $\mathrm{eu}$-weight space decomposition 
$$M=\bigoplus M_\lambda,$$ we define $$DM=\bigoplus M_\lambda^\vee,$$ with $H_c$-action given by the formula
$$(hf)(m)=f(\overline{h} m) \quad \hbox{for $f \in DM$, $h \in H_c$, and $m \in M$.}$$

\subsection{Homology of unitary representations} The following theorem, which relies crucially on results of  Huang-Wong \cite{HuWo} (see also Ciubotaru \cite{Ciu}) might be thought of as the analog of the Hodge decomposition theorem for unitary representations of rational Cherednik algebras. It allows us to convert information about the graded character into information about the Kazhdan-Lusztig character.

\begin{theorem} \label{hodge decomposition} Suppose $L$ is a unitary $H_c$-module. Then as graded $\CC W$-modules
$$H_i(V^*,L) \cong \bigoplus_{c_F=a+i} (L_a \otimes \Lambda^i V^*)_F$$ and
$$\mathrm{dim} \left( \mathrm{Ext}^i (\Delta_c(F),L) \right)=\mathrm{dim} \left( \mathrm{Hom}_{\CC W} (F,H^i(V^*,L) \right).$$
\end{theorem}

This uses the duality of $\OO_c$ together with the results of \cite{HuWo} and \cite{Ciu} on Dirac cohomomlogy; see Theorem 1.3 of \cite{Gri4}.

\section{The cyclotomic groups and their degenerate affine Hecke algebra} \label{cyclotomic}

\subsection{Summary} In this section we present a number of results for the groups $G(\ell,1,n)$ that are parallel to well-known facts for the symmetric group. To the extent that the proofs are also parallel they will be omitted.

\subsection{The group $G(\ell,1,n)$.} The group $G(\ell,1,n)$ consists of all $n$ by $n$ matrices with exactly one non-zero entry in each row and each column, and such that the non-zero entries are $\ell$th roots of $1$. It acts on $\CC^n$. As such, the set of reflections in $G(\ell,1,n)$ is
$$R=\{ \zeta_i^j \ | \ 1 \leq i \leq n \ 1 \leq j \leq \ell -1 \} \cup \{\zeta_i^k s_{ij} \zeta_i^{-k} \ | \ 1 \leq i < j \leq n, \ 0 \leq k \leq \ell-1 \},$$ where we write $\zeta$ for a fixed primitive $\ell$th root of $1$, the notation $\zeta_i$ stands for a diagonal matrix with $1$'s on the diagonal except in position $i$ where it has a $\zeta$, and $s_{ij}$ is the transposition interchanging $i$ and $j$.

 When $\ell$ is fixed and it will not cause confusion, we may write $G_n=G(\ell,1,n)$ for this group. In the same spirit we will use the abbreviations
$$\mathrm{Ind}^n_m=\mathrm{Ind}^{\CC G_n}_{\CC G_m}, \quad \mathrm{Res}^n_m=\mathrm{Res}^{\CC G_n}_{\CC G_m}, \quad \mathrm{Ind}^{m+n}_{m,n}=\mathrm{Ind}^{\CC G_{m+n}}_{\CC (G_m \times G_n)}, \quad \text{and} \quad \mathrm{Res}^{m+n}_{m,n}=\mathrm{Res}^{\CC G_{m+n}}_{\CC (G_m \times G_n)}$$ for the appropriate functors of induction and restriction.

\subsection{Jucys-Murphy-Young elements} Just as in the case of the symmetric group, we wish to study the groups $G_n$ in a family as $n$ varies. We have the following analog of Jucys-Murphy-Young elements: the $G_{j-1}$-orbit sum of $s_{j-1}=(j-1,j)$ is
$$\phi_j=\sum_{\substack{1 \leq i < j \\ 0 \leq k \leq \ell-1}} \zeta_i^k s_{ij} \zeta_i^{-k}=\sum_{\substack{1 \leq i < j \\ 0 \leq k \leq \ell-1}} \zeta_j^k s_{ij} \zeta_j^{-k}.$$ The element $\zeta_j$ commutes with $\phi_j$ and $G_{j-1}$. Modulo its center, the centralizer of $\CC G_{j-1}$ in $\CC G_j$ is generated by $\zeta_j$ and $\phi_j$ so that the branching rule is multiplicity free. Moreover, $\phi_j$ satisfies the relation
$$s_j \phi_j s_j=\sum_{\substack{1 \leq i < j \\ 0 \leq k \leq \ell-1}} \zeta_i^k s_{i,j+1} \zeta_i^{-k}=\phi_{j+1}-\sum_{ 0 \leq k \leq \ell-1} \zeta_j^k s_j \zeta_j^{-k},$$ so that
$$s_j \phi_j=\phi_{j+1} s_j- \sum_{0 \leq k \leq \ell-1} \zeta_j^k \zeta_{j+1}^{-k}.$$ 

\subsection{Intertwiners} Define the \emph{intertwining operators} $\tau_i$ by 
$$\tau_i=s_i+(\phi_i-\phi_{i+1})^{-1} \pi_i$$ for $1 \leq i \leq n-1$, where
$$\pi_i=\sum_{0 \leq k \leq \ell-1} \zeta_i^k \zeta_{i+1}^{-k}.$$ The operator $\tau_i$ is well-defined on each $\CC G(\ell,1,n)$-module, since $\phi_i-\phi_{i+1}$ is invertible on vectors $v$ with $\pi_i v \neq 0$. Moreover they satisfy braid relations and
$$\phi_j \tau_i=\tau_i \phi_{s_i(j)} \quad \hbox{for all $1 \leq i \leq n-1$ and $1 \leq j \leq n$.}$$ They may thus be used to create new eigenvectors for the operators $\phi_i$. 

\subsection{Young's semi-normal form} \label{seminormal} The material of this section is a straightforward adaptation of the standard story for the symmetric group case (see for instance chapter one of \cite{Kle}) to the groups $G(\ell,1,n)$.

The commuting elements $\phi_1,\dots,\phi_n,\zeta_1,\dots,\zeta_n$ act with one-dimensional common eigenspaces on each irreducible $\CC G_n$-module. The eigenvalues may be described as follows: for each $\ell$-partition $\lambda=(\lambda^0,\lambda^1,\dots,\lambda^{\ell-1})$ of $n$ there is an irreducible $\CC G(\ell,1,n)$-module $S^\lambda$ determined up to isomorphism by the property that it has a basis $v_T$ indexed by $T \in \mathrm{SYT}(\lambda)$ such that
$$\phi_i v_T=\ell \mathrm{ct}(T^{-1}(i)) v_T \quad \text{and} \quad \zeta_i v_T=\zeta^{\beta(T^{-1}(i))} v_T \quad \hbox{for all $1 \leq i \leq n$.}$$ We will normalize this basis in such a way that the structure constants for the $G(\ell,1,n)$-action are contained in the ring $\QQ[\zeta]$. To do this we write $T(\lambda)$ for the standard Young tableau defined by filling the boxes of $\lambda$ left to right and top to bottom in order with the numbers $1,2,\dots,n$, where we think of $\lambda^0$ as being to the left of $\lambda^1$ which is to the left of $\lambda^2$ and so on. Every other standard Young tableau is of the form $T=w T(\lambda)$ for a (unique) permutation $w \in S_n$, and we define the \emph{length} $\ell(T)$ of $T$ to be
$$\ell(T)=\ell(w) \quad \text{if} \quad T=w T(\lambda).$$ As in, for example, \cite{Kle}, for the case $\ell=1$, we now fix any eigenvector $v_{T(\lambda)}$ corresponding to $T(\lambda)$, and define
$$v_T=\pi_T(w v_{T(\lambda)}) \quad \hbox{if $T=w T(\lambda)$.}$$ Here $\pi_T$ is the projection on the eigenspace for $T$ whose kernel is the sum of the other eigenspaces.

With this choice of normalization we now have the formulas
$$\tau_i v_T=\begin{cases} 0 \quad \hbox{if $s_i T$ is not standard,} \\
v_{s_i T} \quad \hbox{if $s_i T$ is a standard Young tableau and $\ell(s_i T) > \ell(T)$, or if $\beta(T^{-1}(i)) \neq \beta(T^{-1}(i+1))$, and} \\ \frac{(\mathrm{ct}(T^{-1}(i))-\mathrm{ct}(T^{-1}(i+1)))^2-\ell^2}{(\mathrm{ct}(T^{-1}(i))-\mathrm{ct}(T^{-1}(i+1)))^2} v_{s_i T} \quad \text{else.} \end{cases}$$

These formulas imply:

$$s_i v_T=v_T \quad \hbox{if $i$ appears in the box to the left of $i+1$ in $T$,}$$
$$s_i v_T=-v_T \quad \hbox{if $i$ appears in the box above $i+1$ in $T$,}$$
$$s_i v_T=v_{s_i T} \quad \hbox{if $\beta(T^{-1}(i)) \neq \beta(T^{-1}(i+1))$,}$$ and defining $a=\mathrm{ct}(T^{-1}(i))-\mathrm{ct}(T^{-1}(i+1))$, if $s_i T$ is a standard Young tableau
$$s_i v_T=v_{s_i T}-a^{-1} \ell v_T \quad \hbox{if $\beta(T^{-1}(i))=\beta(T^{-1}(i+1))$ and $\ell(s_i T) > \ell(T)$,}$$ and 
$$s_i v_T=(1-a^{-2} \ell^2) v_{s_i T}+a^{-1} \ell v_T \quad \hbox{if $\beta(T^{-1}(i))=\beta(T^{-1}(i+1))$ and $\ell(s_i T) < \ell(T)$.}$$

\subsection{The degenerate affine Hecke algebra of type $G(\ell,1,n)$} The degenerate affine Hecke algebra of type $G(\ell,1,n)$ is the algebra $H_{\ell,n}$ generated by $\CC[u_1,\dots,u_n]$ and the group $G(\ell,1,n)$, subject to the relations
$$\zeta_i u_j=u_j \zeta_i \quad \hbox{for all $i,j$,} \quad s_i u_j=u_j s_i \quad \hbox{if $j \neq i,i+1$} \quad \text{and} \quad s_i u_i=u_{i+1} s_i-\sum_{0 \leq k \leq \ell-1} \zeta_i^k \zeta_{i+1}^{-k}.$$

\subsection{Automorphisms of $H_{\ell,n}$} There are two types of automorphisms of $H_{\ell,n}$ that we will use. First, we fix a complex number $\alpha$, and define the automorphism $t_\alpha$ of $H_{\ell,n}$ on generators by the rules
$$t_\alpha(u_i)=u_i+\alpha, \quad t_\alpha(\zeta_i)=\zeta_i, \quad \text{and} \quad t_\alpha(s_i)=s_i.$$

Second, there is an automorphism $\rho$ of $H_{\ell,n}$ given on generators by
$$\rho(u_i)=-u_{n-i+1}, \quad \rho(\zeta_i)=\zeta_{n-i+1}, \quad \text{and} \quad \rho(s_i)=s_{n-i}.$$ 

We note that both of these automorphisms preserve the group algebra $\CC G_n$, and in fact are inner automorphisms of it. Given an $H_{\ell.n}$-module $M$ and an automorphism $a$ of $H_{\ell,n}$ we write $M^a$ for the $H_{\ell,n}$-module which is equal to $M$ as an abelian group, and with $H_{\ell,n}$-action defined by
$$h \cdot m=a(h) m \quad \hbox{for $h \in H_{\ell,n}$ and $m \in M$.}$$ If $a=\rho$ or $a=t_\alpha$ then $M^a$ is isomorphic to $M$ as a $\CC G_n$-module.

\subsection{$H_{\ell,n}$ modules via branching for $G_n$} \label{branching}

Let $m$ be a non-negative  integer. There is a map $$H_{\ell,n} \longrightarrow \CC G_{m+n} \quad \text{given by} \quad u_i \mapsto \phi_{m+i} \quad s_j \mapsto s_{m+j} \quad  \text{and} \quad \zeta_i \mapsto \zeta_{m+i}$$ for $1 \leq i \leq n$ and $1 \leq j \leq n-1$. The image of this map is contained in the centralizer of $G_m$ in $\CC G_{m+n}$, so that $H_{\ell,n}$ acts on the module
$$S^{\lambda \setminus \mu}=\mathrm{Hom}_{G_m}(S^\mu, \mathrm{Res}^{m+n}_m(S^\lambda))$$ for all pairs $\lambda,\mu$ of $\ell$-partitions, where $\lambda$ is an $\ell$-partition of $m+n$ and $\mu$ is an $\ell$-partition of $m$. It follows from Young's orthogonal form (\ref{seminormal}) that $S^{\lambda \setminus \mu}=0$ unless $\mu \subseteq \lambda$, in which case $S^{\lambda \setminus \mu}$ has a basis indexed by the set of standard Young tableaux on the skew diagram $\lambda \setminus \mu$, which we now describe.

Given $T \in \mathrm{SYT}(\mu)$ and $U \in \mathrm{SYT}(\lambda \setminus \mu)$ we define $T \cup U \in \mathrm{SYT}(\lambda)$ by
$$T \cup U (b)=T(b) \quad \hbox{if $b \in \mu$ and} \quad T \cup U(b)=U(b)+m \quad \hbox{if $b \in \lambda \setminus \mu$.}$$ We now define $$\psi_U \in S^{\lambda \setminus \mu}=\mathrm{Hom}_{\CC G(r,1,m)}(S^\mu, \mathrm{Res}^{m+n}_m(S^\lambda))$$ by the formula
$$\psi_U(v_T)=v_{T \cup U}.$$ These $\psi_U$ are a basis of $S^{\lambda \setminus \mu}$,
$$S^{\lambda \setminus \mu}=\CC\{\psi_U \ | \ U \in \mathrm{SYT}(\lambda \setminus \mu) \}.$$

\subsection{$H_{\ell,n}$-modules from general skew shapes}

Suppose now that $D \subseteq \RR^2 \times \ZZ/\ell$ is a skew shape with connected components $D_1,\dots,D_k$. After diagonal slides, we may assume that each $D_i$ is such that for any $(x,y) \in D_i$, we have $y_i \in \ZZ$, and moreover the set of $y$-coordinates of distinct $D_i$'s are disjoint. We may choose (non-unique) $\alpha_1,\dots,\alpha_k \in \CC$ and integral skew shapes $\lambda_1 \setminus \mu_1,\dots, \lambda_k \setminus \mu_k$ such that
$$D_i=\lambda_i \setminus \mu_i+(\alpha_i,0),$$ so that their union is disjoint and a skew shape,
$$\lambda \setminus \mu=\coprod \lambda_i \setminus \mu_i,$$ and so that $\lambda_1 \setminus \mu_1,\dots, \lambda_k \setminus \mu_k$ are the connected components of their disjoint union $\lambda \setminus \mu$. We then \emph{define} 
$$S^D=\mathrm{Ind}( \otimes_i (S^{\lambda_i \setminus \mu_i} )^{t_{\alpha_i}}).$$ Up to isomorphism, this representation $S^D$ is independent of the choices made in its construction (this follows e.g. from Young's orthogonal form for skew shapes), and moreover its restriction to $G_n$ is isomorphic to $S^{\lambda \setminus \mu}$ (since the automorphisms $t_\alpha$ are the identity of $G_n$).

\subsection{Littlewood-Richardson numbers}
There are isomorphisms
\begin{align*} \mathrm{Hom}_{G_n}(S^\nu,\mathrm{Hom}_{G_m}(S^\mu, \mathrm{Res}^{m+n}_m(S^\lambda)))& \cong \mathrm{Hom}_{G_m \times G_n} (S^\mu \otimes S^\nu, \mathrm{Res}^{m+n}_{m,n} S^\lambda) \\ &\cong \mathrm{Hom}_{G_{m+n}}(\mathrm{Ind}^{m+n}_{m,n} (S^\mu \otimes S^\nu),S^\lambda),\end{align*} the first induced by $\otimes$-$\mathrm{Hom}$ adjunction and the second given by Frobenius reciprocity. We will write $$c^\lambda_{\mu \nu}=\mathrm{dim}_\CC(\mathrm{Hom}_{G_{m+n}}(\mathrm{Ind}^{m+n}_{m,n} (S^\mu \otimes S^\nu),S^\lambda)),$$ and refer to $c^\lambda_{\mu \nu}$ as a \emph{(cyclotomic) Littlewood-Richardson number}. Likewise, defining
$$c^{\lambda \setminus \mu}_\nu=\mathrm{dim}_\CC(\mathrm{Hom}_{G_n}(S^\nu,\mathrm{Hom}_{G_m}(S^\mu, \mathrm{Res}^{m+n}_m(S^\lambda))))$$ we have $c^{\lambda \setminus \mu}_\nu=c^\lambda_{\mu \nu}$. If now $D$ is any skew diagram, integral or not, such that $\lambda \setminus \mu$ may be obtained from $D$ by translating its connected components without merging any of them, then $S^D \cong S^{\lambda \setminus \mu}$ as $G_n$-modules, and hence defining
$$c^D_\nu=c^{\lambda \setminus \mu}_\nu$$ we have an isomorphism of $G_n$-modules:
$$S^D \cong \bigoplus_\nu (S^\nu)^{\oplus c^D_\nu}.$$

\subsection{Relationship with classical Littlewood-Richardson numbers}

Let $\lambda=(\lambda^0,\dots,\lambda^{\ell-1}), \mu=(\mu^0,\dots,\mu^{\ell-1})$, and $\nu=(\nu^0,\dots,\nu^{\ell-1})$ be $\ell$-partitions. The Littlewood-Richardson number $c^\nu_{\lambda \mu}$ may be expressed as a product
$$c^\nu_{\lambda \mu}=\prod_{j=0}^{\ell-1} c^{\nu_j}_{\lambda_j \mu_j}.$$ This follows from standard facts about induction and tensor products, and the realization of $S^\lambda$ as
$$S^\lambda=\mathrm{Ind}^{G_n}_{G_{m_0} \times G_{m_1} \times \cdots \times G_{m_{\ell-1}}} (S^{\lambda^0} \otimes S^{\lambda^1} \otimes \cdots \otimes S^{\lambda^{r-1}}),$$ where $m_i=|\lambda_i|$ and $G_{m_i}$ acts on the Specht module $S^{\lambda^i}$ for $S_{m_i}$ via the surjection $\CC G_{m_i} \to \CC S_{m_i}$ that is the identity on $S_{m_i}$ and sends $\zeta_j$ to $\zeta^i$ for all $j$.

\subsection{Tensor product with the exterior powers $\Lambda^i (V^*)$} The Littlewood-Richardson numbers enter into our calculations in a second way, via the calculation of the tensor products $S^\nu \otimes \Lambda^i (V^*)$ as $G_n$-modules, where as usual we write $V=\CC^n$ for the defining representation of $G_n$. If $\lambda=(\lambda^0,\lambda^1,\dots,\lambda^{\ell-1})$ is an $\ell$-partition of $n$, then we have
$$S^\lambda \otimes \Lambda^n(V^*)=S^{\lambda^t},$$ where we write $\lambda^t$ for the $\ell$-partition
$$\lambda^t=((\lambda^1)^t,(\lambda^2)^t,\dots,(\lambda^{\ell-1})^t,(\lambda^0)^t)$$ obtained from $\lambda$ by cycling its components one spot to the left and transposing them all.  When $n$ is clear from context, we will write simply $\mathrm{det}^{-1}$ for the one-dimensional character of $G_n$ acting on $\Lambda^n(V^*)$. 

In general, we observe that $\Lambda^i (V^*)$ contains the vector $v_{n-i+1} \wedge v_{n-i+2} \wedge \cdots \wedge v_n$, which is fixed by $G_{n-i}$ and transforms like $\mathrm{det}^{-1}$ under $G_i$ embedded in $G_n$ via the last $i$ coordinates, and is therefore induced from a one-dimensional representation
$$\Lambda^i(V^*)=\mathrm{Ind}^{G_n}_{G_{n-i} \times G_i} (1 \times \mathrm{det}^{-1}).$$ Computing the tensor product of this representation with $S^\nu$ may therefore be achieved as follows: by the projection formula we have
$$S^\nu \otimes \Lambda^i(V^*) \cong S^\nu \otimes \mathrm{Ind}^{G_n}_{G_{n-i} \times G_i} (1 \times \mathrm{det}^{-1}) \cong\mathrm{Ind}^{G_n}_{G_{n-i} \times G_i} (\mathrm{Res}^n_{n-i,i}(S^\nu) \otimes (1 \times \mathrm{det}^{-1}))$$ and hence
\begin{align*} \mathrm{Hom}_{G_n} &(S^\mu, S^\nu \otimes \Lambda^i(V^*)) \cong \mathrm{Hom}_{G_n} (S^\mu,\mathrm{Ind}^{n}_{n-i,i}(\mathrm{Res}^{n}_{n-i,i}(S^\nu) \otimes (1 \times \mathrm{det}^{-1}))) \\ & \cong \mathrm{Hom}_{G_{n-i,i}} (\mathrm{Res}^n_{n-i,i}(S^\mu),\mathrm{Res}^{n}_{n-i,i}(S^\nu) \otimes (1 \times \mathrm{det}^{-1}))). \end{align*} Taking dimensions gives
\begin{equation} \label{tensor formula}
\mathrm{dim}(  \mathrm{Hom}_{G_n} (S^\mu, S^\nu \otimes \Lambda^i(V^*)) )=\sum_{\substack{\eta \vdash n-i \\ \chi \vdash i}} c^\nu_{\eta \chi} c^\mu_{\eta \chi^t}.
\end{equation}

\subsection{Classification of irreducible $\mathfrak{u}$-diagonalizable $H_{\ell,n}$-modules}

\begin{lemma} \label{content sequence lemma}
Let $(a_1,\dots,a_n,\zeta^{b_1},\dots,\zeta^{b_n})$ be a sequence satisfying the property: if $i<j$ with $a_i=a_j$ and $b_i=b_j$ mod $\ell$, then there are $i<k,m<j$ with $b_k=b_m=b_i$ mod $\ell$ and $$a_k=a_i+\ell, \quad a_m=a_i-\ell.$$ Then there is a skew shape $D$ and a standard Young tableau $T$ of shape $D$ satisfying 
$$\ell \mathrm{ct}(T^{-1}(i))=a_i \quad \text{and} \quad \beta(T^{-1}(i))=b_i \quad \hbox{for $1 \leq i \leq n$.}$$ Moreover, $T$ and $D$ are unique up to diagonal slides of their connected components.
\end{lemma}
\begin{proof}
Induct on $n$. Thus by induction the sequence $(a_1,\dots,a_{n-1},\zeta^{b_1},\dots,\zeta^{b_{n-1}})$ possesses a standard Young tableau $T'$ on a skew diagram $D'$, and our hypothesis implies that this $D'$ possesses an  addable box $b$ with $\beta(b)=b_n$ mod $\ell$ and $\ell \mathrm{ct}(b)=a_n$. We obtain $T$ and $D$ by adjoining $b$ to $D'$ and defining $T(b)=n$. 
\end{proof}

We remark that the proof of the lemma amounts to an effective recursion for constructing $T$ and $D$. 

\begin{theorem} \label{diag class}
Let $M$ be an irreducible $\mathfrak{u}$-diagonalizable $H_{\ell,n}$-module and suppose $m \in M$ satisfies
$$u_i m=a_i m \quad \text{and} \quad \zeta_i m =\zeta^{b_i} m \quad \hbox{for $1 \leq i \leq n$.}$$ Then there is a standard Young tableau $T$ on a skew shape $D$ such that $a_i=\ell \mathrm{ct}(T^{-1}(i))$ and $b_i=\beta(T^{-1}(i))$ for $1 \leq i \leq n$ and $M \cong S^D$, and moreover $T$ and $D$ are unique up to diagonal slides of their connected components. 
\end{theorem}
\begin{proof}
One checks that the sequence $(a_1,\dots,a_n,\zeta^{b_1},\dots,\zeta^{b_n})$ satisfies the hypothesis of the preceding lemma as in \cite{Kle} Cor. 2.2.4. This produces $T$ and $D$, and one proves that the $\mathfrak{u}$-module structure of $M$ is the same as that of $S^D$. The analog of Theorem 5.3.1 of \cite{Kle} then shows that $M \cong S^D$.  
\end{proof}

\section{Proof of the first main theorem}

\subsection{Parameters and the reflection representation of the group $G(\ell,1,n)$}  The group $W=G(\ell,1,n)$ acts on $V=\CC^n$ in the obvious way. We will write $\CC[V]=\CC[x_1,\dots,x_n]$ and $\CC[V^*]=\CC[y_1,\dots,y_n]$, where $y_1,\dots,y_n$ is the standard basis of $V=\CC^n$ with dual basis $x_1,\dots,x_n$ of $V^*$. The deformation parameter $c$ is a tuple $c=(c_0,d_0,\dots,d_{\ell-1})$ of real numbers with $d_0+d_1+\cdots+d_{\ell-1}=0$. We define $d_i$ for $i \in \ZZ$ by $d_i=d_j$ if $i=j$ mod $\ell$. 

\subsection{The $G(\ell,1,n)$ Cherednik algebra} Specializing $W=G(\ell,1,n)$, we will simply write $H_c$ for the cyclotomic rational Cherednik algebra. This is generated by two polynomial rings $\CC[x_1,\dots,x_n]$ and $\CC[x_1,\dots,x_n]$ and the group algebra $\CC W$, which acts by automorphisms on the two polynomial rings, subject to the relations
$$w x w^{-1} = w(x) \quad \text{and} \quad w y w^{-1}=w(y) \quad \hbox{for $w \in W$, $x \in V^*$, and $y \in V$,}$$
$$y_i x_i=x_i y_i+1-c_0 \sum_{\substack{ 1 \leq j \neq i \leq n \\ 0 \leq r \leq \ell-1}} \zeta_i^r s_{ij} \zeta_i^{-r}-\sum_{r=0}^{\ell-1} (d_r-d_{r-1}) e_{ir}$$ for $1 \leq i \leq n$, where
$$e_{ir}=\frac{1}{\ell} \sum_{t=0}^{\ell-1} \zeta^{-tr} \zeta_i^t$$ and
$$y_i x_j=x_j y_i+ c_0 \sum_{r=0}^{\ell-1} \zeta^{-r} \zeta_i^r s_{ij} \zeta_i ^{-r}$$ for $1 \leq i \neq j \leq n$.

\subsection{The Dunkl-Opdam subalgebra} We define 
$$z_i=y_i x_i+c_0 \phi_i,$$ where $\phi_i$ is the $i$th Jucys-Murphy-Young element for the group $G(\ell,1,n)$. Dunkl-Opdam \cite{DuOp} introduced these elements of $H_c$ in order to generalize the definition of non-symmetric Jack polynomials to the group $G(\ell,1,n)$, and proved that $z_i$ and $z_j$ commute for all $i$ and $j$. Moreover, they satisfy the relations
$$\zeta_i z_j=z_j \zeta_i \quad \hbox{for all $1 \leq i,j \leq n$,} \quad s_i z_j=z_j s_i \quad \hbox{for $j \neq i, i+1$, and} \quad  s_i z_i=z_{i+1} s_i-c_0 \sum_{0 \leq k \leq \ell-1} \zeta_j^k \zeta_{j+1}^{-k}.$$ The \emph{Dunkl-Opdam} subalgebra  of $H_c$ is the subalgebra $\ttt$ generated by $z_1,\dots,z_n$ and $\zeta_1,\dots,\zeta_n$.

In \cite{Gri2}, the second author computed the spectrum of this subalgebra $\ttt$ on each standard module $\Delta_c(\lambda)$, and defined the representation-valued version of the non-symmetric Jack polynomials to be the $\ttt$-eigenfunctions in $\Delta_c(\lambda)$. These play the central role in the classification \cite{Gri3} of the unitary representations in $\OO_c$.

\subsection{Intertwiners} In \cite{Gri} the second author introduced the following intertwining operators:
$$\sigma_i=s_i+c_0 (z_i-z_{i+1})^{-1} \pi_i \quad \hbox{for $1 \leq i \leq n-1$,}$$
$$\Phi=x_n s_{n-1} s_{n-2} \cdots s_1,$$ and
$$\Psi=y_1 s_1 s_2 \cdots s_{n-1}.$$ 

\subsection{$H_{\ell,n}$ as a subalgebra of $H_c$} The map determined by
$$u_i \mapsto \frac{1}{c_0} z_i \quad s_i \mapsto s_i \quad \zeta_i \mapsto \zeta_i$$ is an injection of $H_{\ell,n}$ into $H_c$. Via this injection, $u_i$ acts on $f_{P,Q}$ by 
$$u_i f_{P,Q}=\frac{1}{c_0}( Q(P^{-1}(i))+1-(d_{\beta(P^{-1}(i))}-d_{\beta(P^{-1}(i))-Q(P^{-1}(i))-1}))-\ell \mathrm{ct}(P^{-1}(i)).$$ We will use this formula to identify the isotype of the $H_{\ell,n}$-modules arising upon restricting $L_c(\lambda)$ to $H_{\ell,n}$. 

\subsection{The diagonalizable case}

We begin by proving the first part of Theorem \ref{main}, which is a consequence of the following more precise version. First, given a skew diagram $D$ we define a skew diagram $D^r$, the \emph{reverse} of $D$, as follows. Twisting $S^D$ by the automorphism $\rho$ of $H_{\ell,n}$ we obtain another $\mathfrak{u}$-diagonalizable module $(S^D)^\rho$, and $D^r$ is the skew diagram with $S^{D^r} \cong (S^D)^\rho$.  

\begin{theorem} \label{main version}
Let $L_c(\lambda)$ be a $\ttt$-diagonalizable $H_c$-module and let $d$ be a positive integer. Then as a $H_{\ell,n}$-module, the degree $c_\lambda+d$ part of $L_c(\lambda)$ is semisimple and isomorphic to the direct sum
$$L_c(\lambda)_{c_\lambda+d} \cong \bigoplus_{Q \in \mathrm{Tab}_c(\lambda), |Q|=d} S^{s_c(Q)^r}.$$ As a consequence of this, part (1) of Theorem \ref{main} holds.
\end{theorem}
\begin{proof}
We begin by proving that for $Q \in \mathrm{Tab}_c(\lambda)$ fixed, the span
$$L_Q=\CC\{f_{P,Q} \ | \ (P,Q) \in \Gamma_c(\lambda) \}$$ is an irreducible $H_{\ell,n}$-module. This granted, there is a (unique up to diagonal slides of connected components) skew diagram $D$ and a standard Young tableau $T$  on $D$ with
$$ \mathrm{ct}(T^{-1}(i))=\frac{1}{\ell c_0}( Q(P^{-1}(i))+1-(d_{\beta(P^{-1}(i))}-d_{\beta(P^{-1}(i))-Q(P^{-1}(i))-1}))-\mathrm{ct}(P^{-1}(i)).$$ It follows from this, the definition of $s_c(Q)$, and Theorem \ref{diag class} that $L_Q$ is isomorphic to $S^{s_c(Q)^r}$ as an $H_{\ell,n}$-module, which proves the result.

To establish irreducibility we use Lemma 7.4 of \cite{Gri2}. This lemma, translated into the notation we use here, shows that given $P,P'$ with $(P,Q), (P',Q) \in \Gamma_c(\lambda)$, there is a sequence of simple transpositions $s_{i_1},\dots,s_{i_p}$ such that $P'=s_{i_1} \cdots s_{i_p} P$ and $(s_{i_j} \cdots s_{i_p} P,Q) \in \Gamma_c(\lambda)$ for all $1 \leq j \leq p$. This in turn shows that the $H_{\ell,n}$-submodule of $L_c(\lambda)$ generated by $f_{P,Q}$ is $L_Q$; together with the fact that any $H_{\ell,n}$-submodule of $L_Q$ must contain some weight vector this finishes the proof.
\end{proof}

The following corollary proves the first part of Theorem \ref{main}.

\begin{corollary}
Suppose $L_c(\lambda)$ is $\ttt$-diagonalizable. As a $\CC G_n$-module, the degree $c_\lambda+d$ part of $L_c(\lambda)$ is
$$L_c(\lambda)_{c_\lambda+d} \cong \bigoplus_{\substack{Q \in \mathrm{Tab}_c(\lambda), \ |Q|=d \\ \mu \in P_{\ell,n}}} (S^\mu)^{\oplus c^{s_c(Q)}_\mu}.$$
\end{corollary}
\begin{proof}
Twisting the representation $S^{s_c(Q)^r}$ by the automorphism $\rho$ of $H_{\ell,n}$ shows that as a $\CC G_n$-module, $L_Q$ is isomorphic to $S^{s_c(Q)}$, which proves the corollary.
\end{proof}

Finally we deduce the second part of Theorem \ref{main}.

\begin{corollary}
Suppose $L_c(\lambda)$ is unitary. Let $i$ be a non-negative integer and let $\mu$ be an $\ell$-partition of $n$. Then
$$\mathrm{dim}(\mathrm{Ext}^i(\Delta_c(\mu),L_c(\lambda))=\sum c^{s_c(Q)}_\nu c^\nu_{\eta \chi} c^\mu_{\eta \chi^t}.$$
\end{corollary}
\begin{proof}
Since $L_c(\lambda)$ is unitary we may apply Theorem \ref{hodge decomposition} and \eqref{tensor formula} in combination with the preceding corollary. The formula follows.
\end{proof}

\section{Maps between standardard modules} \label{BGG}

In this section, we define maps between certain standard modules. The
graded poset $\Pnk$ is essential for our
definition.
\subsection{The poset $\Pnk$}\label{subsec:poset}
\subsubsection{Partitions} We identify an integer partition
$\lambda=(\lambda_1,\ldots,\lambda_r)$, where
$\lambda_1\geq\lambda_2\dots\geq\lambda_r>0$, with its Young
diagram. \omitt{Its Young diagram is a left-justified array of
boxes/cells, with $\lambda_i$ boxes in the $i^{\text{th}}$ row from
the top.} We will often refer to cells $B$ in the diagram of $\lambda$
and simply write $B\in\lambda$. A {\em tableau} $T$ of shape $\lambda$
is a diagram of $\lambda$ where the cells have nonnegative integer
entries; $T_{ij}$ is the entry in row $i$, column $j$. A {\em
violation} in $T$ is either a pair of entries $T_{ij}$ and $T_{i,j+1}$
such that $T_{ij} > T_{i,j+1}$ (row violation) or pair of entries
$T_{ij}$ and $T_{i+1,j}$ such that $T_{ij} > T_{i+1,j}$ (column
violation). If $\lambda$ is a partition of $n$ and $T$ is a tableau
with no violations and entries $1,2,\ldots,n$, then $T$ is a {\em
standard Young tableau of shape $\lambda$}, written
$T\in\syt(\lambda)$. The {\em content} of the cell $B$ in row $i$ and
column $j$ is $j-i$ and is written $\ctBox{B}$. For $T\in\syt(\lambda)$,
$\ct{h}{T}$ is the content of the cell in $T$ containing $h$. As
usual, when $\gamma$ and $\lambda$ are both partitions, $\gMinusl$ is
the collection of cells in $\gamma$ which are not in $\lambda$ and
$\gamma\cap\lambda$ are the cells in both $\gamma$ and
$\lambda$.

A cell $(x,y)$ in $\lambda$ is in region $\reg$ with residue $r$ if
$y-x=k\reg +r$, where $0\leq r<k$. In general, the region and residue
of a cell $B$ will be denoted by $\reg_B$ and $r_B$ respectively. A
region which we number $\reg$ would be numbered $\reg-1$
in \cite{GKS}.  Additionally, we define the diagonal $D_B$ of a cell
$B$: $$D_B=\{C\in\lambda|\ctBox{C}=\ctBox{B}\}.$$

\subsubsection{Elements of the poset}
\begin{definition}\label{def:abacus}
 The {\em $k$-abacus diagram} of a partition $\lambda$ whose first
 part $\lambda_1$ is strictly less than $k$ is as in Definition 4.1 of
 \cite{BGS}, which we reproduce here. Given a partition $\lambda$
 whose first part is less than $k$, place its Young
 diagram in the fourth quadrant of the plane. Its upper left corner
 should be at (0,0). The $k$-abacus has $k$ runners, labeled
 $0,1,\ldots,k-1$ from the left, and beads and spaces are in rows on
 the abacus: row 0 at the bottom, row 1 above it, and so on. Position
 $m$ is in row $i$, runner $j$ if $m=ik+j$, where $0\leq j<k$.

We place beads on the runners by tracing the silhouette of the diagram
of $\lambda$. Start tracing at $(k-1,0)$ and continue until reaching
the $y$-axis, producing a series of ``down steps'' and ``left steps.''
We now read across the positions in the abacus from left to right and
bottom to top, leaving an empty space for each down step, and a bead
for each left step. There is a bead in position $m$ of the abacus if
and only if the $m^{\text{th}}$ step is a left step. Since there are
$k$ runners and $k-1$ left steps, there is a runner with no beads on
it.
\end{definition}

\begin{figure}[h]
\begin{tikzpicture}
  \begin{scope}
\def\s{.75}
\def\xa{-1.2}
\def\xb{5.1}
\def\ya{-5.1}
\def\yb{1.2}
\draw[step=\s cm, gray, very thin] (\xa*\s,\ya*\s) grid (\xb*\s,\yb*\s);
\draw[thick] (\xa*\s,0)--(\xb*\s,0);
\draw[thick] (0,\ya*\s)--(0,\yb*\s);
\draw[ultra thick,blue] (4*\s,0)--(4*\s,-3*\s)--(3*\s,-3*\s)--(3*\s,-4*\s)--(0,-4*\s)--(0,\ya*\s);
\end{scope}

\begin{scope}[scale=.3,shift={(17,-9)}]
\def\s{1}
\draw[shift={(0,0)},thick,blue] (0,0) to (0, 10);
\draw[shift={(3,0)},thick,blue] (0,0) to (0, 10);
\draw[shift={(6,0)},thick,blue] (0,0) to (0, 10);
\draw[shift={(9,0)},thick,blue] (0,0) to (0, 10);
\draw[shift={(12,0)},thick,blue] (0,0) to (0, 10);

\node[thick,black] at (0,0) {x};
\node[thick,black] at (3,0) {x};
\node[thick,black] at (6,0) {x};
\shade[shading=ball, ball color = red] (9,0) circle (1) node {};
\node[thick,black] at (12,0) {x};

\shade[shading=ball, ball color = red] (0,2) circle (1) node {};
\shade[shading=ball, ball color = red] (3,2) circle (1) node {};
\shade[shading=ball, ball color = red] (6,2) circle (1) node {};
\node[thick,black] at (9,2) {x};
\node[thick,black] at (12,2) {x};

\node[thick,black] at (0,4) {x};
\node[thick,black] at (3,4) {x};
\node[thick,black] at (6,4) {x};
\node[thick,black] at (9,4) {x};
\node[thick,black] at (12,4) {x};

\node[thick,black] at (0,6) {x};
\node[thick,black] at (3,6) {x};
\node[thick,black] at (6,6) {x};
\node[thick,black] at (9,6) {x};
\node[thick,black] at (12,6) {x};

\node[thick,black] at (0,8) {x};
\node[thick,black] at (3,8) {x};
\node[thick,black] at (6,8) {x};
\node[thick,black] at (9,8) {x};
\node[thick,black] at (12,8) {x};
\node[thick,black] at (0,10) {x};
\end{scope}

\end{tikzpicture}

\caption{The partition $\bgspart{15}{5}=4^33$ and its abacus.}
\end{figure}

For positive integers $n>k>1$, let $q$ and $r$ be given by
$n=q(k-1)+r$ and $0\leq r<k-1$. Then define the partition
$\bgspart{n}{k}$ as $(k-1,\ldots,k-1,r)$, where there are $q$ parts of
size $k-1$. This is the ``initial'' partition of the set of vertices
of our poset. We begin with the $k$-abacus of $\bgspart{n}{k}$. Let
$\bgsset$ be the set of all partitions obtained as a composition of
abacus moves of the following type. We may move any bead down a few rows
as long as we move a bead in another column up the same number of rows. This set is
defined in \cite{BGS} in Notation 4.4, where here
$\lambda=\bgspart{n}{k}$. The abacus for the partition
$\bgspart{n}{k}$ will have at most one bead per runner, so that all
the partitions in $\bgsset$ will also have at most one bead per runner.

\subsubsection{Partial order}
\label{sec:partialOrder}
To describe the partial order on $\bgsset$, write the abacus of a
partition in $\bgsset$ as $(a_0,\ldots,a_{k-1})$, where $a_i$ is the
number of empty spaces below the bead in runner $i$ if runner $i$ has
a bead and $-\infty$ if there is no bead on runner $i$; equivalently, $a_i$ is the row of the bead in row $i$.

\begin{definition}
\label{def:extAba}
  Let $\lambda\in\bgsset$ have abacus $(a_0,\ldots,a_{k-1})$. Extend
  the abacus to have index set $\Z$ by setting $a_{i-k}=a_i+1$ for $a_i\neq
  -\infty$ and call this the {\em extended $k$-abacus} of $\lambda$.
  \end{definition}

  We put an order on $\bgsset$ and call the resulting poset $\Pnk$.
\begin{definition}[Cover relation for $\Pnk$]
\label{def:cover}
  Suppose we have $i$ and $j$ such that 
\begin{enumerate}

\item $j-k<i<j$ 
\item $i\bmod k\neq j\bmod k$ 
\item $a_i>a_j\neq -\infty$ in the extended $k$-abacus of $\lambda$.
\item There is no $h$ such that $i<h<j$ and $a_j\leq
  a_h\leq a_i$. 
\end{enumerate}

Then $\lambda$ is covered by the partition $\gamma$
  with extended
  $k$-abacus $$\ldots,a_{i-1},a_j,a_{i+1},\ldots,a_{j-1},a_i,a_{j+1},\ldots.$$
  The $k$-abacus for $\gamma$ is the same as the $k$-abacus for
  $\lambda$, except that $a_i$ and $a_j$ have been interchanged. We may always assume that $0\leq i <k$. 
   
  \end{definition}
In Figure~\ref{fig:twoParts}, there is an example of partitions $\gamma\gtrdot\lambda$ and in Figure~\ref{fig:twoAbacus}, their corresponding abacuses. The poset $\mc{P}(15,5)$ is in Figure~\ref{fig:poset}.

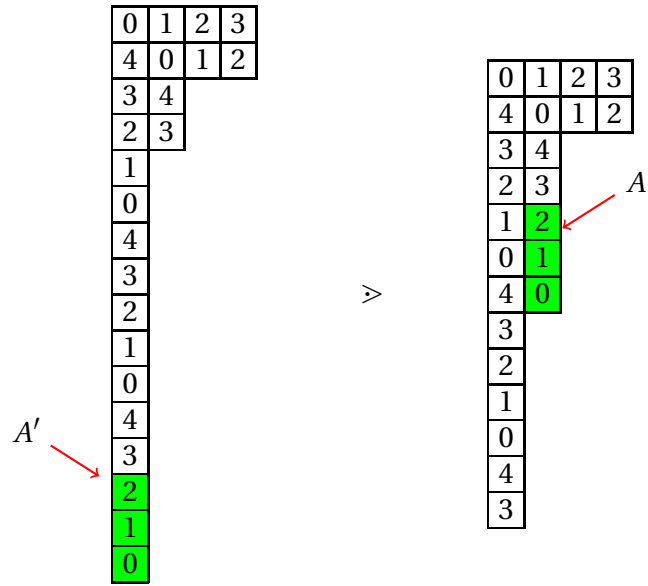
\begin{figure}[h]
 \ytableausetup{boxsize=1.25em}

 \begin{tikzpicture}
 \node (A1) at (-2.1, -1.8){$A'$};
\node (B1) at (-1., -2.5){};
\draw [thick, red,->](A1)--(B1);  
\node at (0,0){\begin{ytableau}
0&1&2&3\\
4&0&1&2\\
3&4\\
2&3\\
1\\
0\\
4\\
3\\
2\\
1\\
0\\
4\\
3\\
*(green)2\\
*(green)1\\
*(green)0
\end{ytableau}

};

\node at (2.5,0){$\gtrdot$};
\node (A) at (6, 1.5){$A$};
\node (B) at (4.9, .8){};
\draw [thick, red,->](A)--(B);
\node at (5,0){\begin{ytableau}
0&1&2&3\\
4&0&1&2\\
3&4\\
2&3\\
1&*(green)2\\
0&*(green)1\\
4&*(green)0\\
3\\
2\\
1\\
0\\
4\\
3
\end{ytableau}
};

\end{tikzpicture}

\caption{A cover relation in the poset $\mc{P}(24,5)$ ($n=24$, $k=5$). The partition $\gamma=4^22^21^{12}$ covers  $\lambda=4^22^51^6$} 
\label{fig:twoParts}
\end{figure}

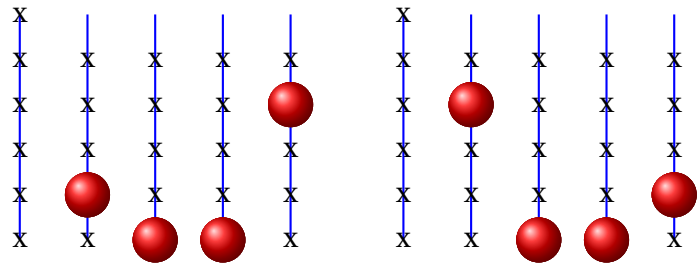
\begin{figure}[h]
\begin{tikzpicture}
  \begin{scope}[scale=.3]
\def\s{1}
\draw[shift={(0,0)},thick,blue] (0,0) to (0, 10);
\draw[shift={(3,0)},thick,blue] (0,0) to (0, 10);
\draw[shift={(6,0)},thick,blue] (0,0) to (0, 10);
\draw[shift={(9,0)},thick,blue] (0,0) to (0, 10);
\draw[shift={(12,0)},thick,blue] (0,0) to (0, 10);
\node[thick,black] at (0,0) {x};
\node[thick,black] at (3,0) {x};
\shade[shading=ball, ball color = red] (6,0) circle (1) node {};
\shade[shading=ball, ball color = red] (9,0) circle (1) node {};
\node[thick,black] at (12,0) {x};
\node[thick,black] at (0,2) {x};
\shade[shading=ball, ball color = red] (3,2) circle (1) node {};
\node[thick,black] at (6,2) {x};
\node[thick,black] at (9,2) {x};
\node[thick,black] at (12,2) {x};
\node[thick,black] at (0,4) {x};
\node[thick,black] at (3,4) {x};
\node[thick,black] at (6,4) {x};
\node[thick,black] at (9,4) {x};
\node[thick,black] at (12,4) {x};
\node[thick,black] at (0,6) {x};
\node[thick,black] at (3,6) {x};
\node[thick,black] at (6,6) {x};
\node[thick,black] at (9,6) {x};
\shade[shading=ball, ball color = red] (12,6) circle (1) node {};
\node[thick,black] at (0,8) {x};
\node[thick,black] at (3,8) {x};
\node[thick,black] at (6,8) {x};
\node[thick,black] at (9,8) {x};
\node[thick,black] at (12,8) {x};
\node[thick,black] at (0,10) {x};
\end{scope}

\begin{scope}[scale=.3,shift={(17,0)}]
\def\s{1}
\draw[shift={(0,0)},thick,blue] (0,0) to (0, 10);
\draw[shift={(3,0)},thick,blue] (0,0) to (0, 10);
\draw[shift={(6,0)},thick,blue] (0,0) to (0, 10);
\draw[shift={(9,0)},thick,blue] (0,0) to (0, 10);
\draw[shift={(12,0)},thick,blue] (0,0) to (0, 10);
\node[thick,black] at (0,0) {x};
\node[thick,black] at (3,0) {x};
\shade[shading=ball, ball color = red] (6,0) circle (1) node {};
\shade[shading=ball, ball color = red] (9,0) circle (1) node {};
\node[thick,black] at (12,0) {x};
\node[thick,black] at (0,2) {x};
\node[thick,black] at (3,2) {x};
\node[thick,black] at (6,2) {x};
\node[thick,black] at (9,2) {x};
\shade[shading=ball, ball color = red] (12,2) circle (1) node {};
\node[thick,black] at (0,4) {x};
\node[thick,black] at (3,4) {x};
\node[thick,black] at (6,4) {x};
\node[thick,black] at (9,4) {x};
\node[thick,black] at (12,4) {x};
\node[thick,black] at (0,6) {x};
\shade[shading=ball, ball color = red] (3,6) circle (1) node {};
\node[thick,black] at (6,6) {x};
\node[thick,black] at (9,6) {x};
\node[thick,black] at (12,6) {x};
\node[thick,black] at (0,8) {x};
\node[thick,black] at (3,8) {x};
\node[thick,black] at (6,8) {x};
\node[thick,black] at (9,8) {x};
\node[thick,black] at (12,8) {x};
\node[thick,black] at (0,10) {x};
\end{scope}

\end{tikzpicture}
\caption{The partitions $\gamma$ and $\lambda$ are as in Figure~\ref{fig:twoParts}. The abacus for $\gamma$ is on the left, for $\lambda$ on the right. Here $i=1$, $a=a_i=3$, $j=4$, and $b=a_j=1$. }
\label{fig:twoAbacus}
\end{figure}

In the remainder of this section, we collect a few facts about the
cover relation and partitions in $\Pnk$. Fix $n$ and $k$ and fix $\gamma\gtrdot\lambda$, both
in $\Pnk$, with extended $k$-abacuses as in
Definition~\ref{def:cover}. Let $\ell=j-i$ and $m=a_i-a_j$. Let $A$
denote the cell at the top of the strip $\lMinusg$ in $\lambda$ and
$A'$ the cell at the top of the strip $\gMinusl$ in $\gamma$. We have
$\reg_A$ and $r_A$ for the region and residue of $A$ and $\reg_{A'}$
and $r_{A'}$ for $A'$.

\begin{claim} \label{claim:singleCol} Using the notation described in previous paragraph, we claim the following:
\begin{enumerate}
\item both $\gMinusl$ and $\lMinusg$ are strips of length $\ell$ and each is contained in a single column;

\item the residues $r_A$ and $r_{A'}$ are equal;

\item 
  
  $$m=\frac{\ct{A}{\lambda}-\ct{A'}{\gamma}}{k};$$
  \item $\ell=|\gMinusl|<k$;
\item $m=\reg_A-\reg_{A'}$; and 

\item $\reg_{B}\leq 0$ for all $B\in\lambda$.

\end{enumerate}
\end{claim}

Claim~\ref{claim:singleCol} follows from the definitions of the cover relation, the map between abacuses and partitions, the fact that $\lambda_1<k$, and well-known facts about abacuses \cite{vL,JK}.

Let $\R'$ be the cells in $\lambda$ which have residue $r_A$ and  
$\R$  the subset of $\R'$ which lie in region $\reg$,
for $\reg_{A'}\leq \reg\leq \reg_A$.

For a cell $B$ in $\lambda$, let $\unders{B}$ be the set of cells in the same column as $B$ which are both
below $B$  and above the next diagonal with
the same residue of $B$, as well as the cell $B$ itself. We have
$|\unders{B}|\leq k$.

\begin{figure}[h]
   \ytableausetup{boxsize=1.25em}

 \begin{tikzpicture}
\begin{scope}
\def\a{5}
\def\b{35}
\def\c{65}
\def\d{95}
\node (A0) at (3., 2.5){Region $0$};
\node (B0) at (1., 2.5){};
\draw [thick, red,->](A0)--(B0);

\node (A1) at (2, 1){Region $-1$};
\node (B1) at (0, 1){};
\draw [thick, red,->](A1)--(B1);

\node (A2) at (1.5, -1){Region $-2$};
\node (B2) at (-.5, -1){};
\draw [thick, red,->](A2)--(B2);

\node (A3) at (1.5, -2.6){Region $-3$};
\node (B3) at (-.5, -2.6){};
\draw [thick, red,->](A3)--(B3);

\node at (0,0){\begin{ytableau}
*(gray!\a)0&*(gray!\a)1&*(gray!\a)2&*(gray!\a)3\\
*(gray!\b)4&*(gray!\a)0&*(gray!\a)1&*(gray!\a)2\\
*(gray!\b)3&*(gray!\b)4\\
*(gray!\b)2&*(gray!\b)3\\
*(gray!\b)1&*(gray!\b)2\\
*(gray!\b)0&*(gray!\b)1\\
*(gray!\c)4&*(gray!\b)0\\
*(gray!\c)3\\
*(gray!\c)2\\
*(gray!\c)1\\
*(gray!\c)0\\
*(gray!\d)4\\
*(gray!\d)3
\end{ytableau}
};
\end{scope}

\begin{scope}[shift={(6,0)}]
\def\li{20}
\def\da{60}
\node (A) at (2, 1){$A$};
\node (B) at (0, 1){};
\draw [thick, red,->](A)--(B);

\node (A1) at (1.5, -1.){$B$};
\node (B1) at (-.5, -1){};
\draw [thick, red,->](A1)--(B1);

\node at (0,0){\begin{ytableau}
0&1&2&3\\
4&0&1&2\\
3&4\\
2&3\\
1&*(gray!\li)2\\
0&*(gray!\li)1\\
4&*(gray!\li)0\\
3\\
*(gray!\da)2\\
*(gray!\da)1\\
*(gray!\da)0\\
*(gray!\da)4\\
*(gray!\da)3
\end{ytableau}
};
\end{scope}

\end{tikzpicture}

\caption{On the left, the regions are labelled for $k=5$. The cells contain the residue mod $5$ of their contents. On the right, $\unders{A}$ is shaded light gray and $\unders{B}$ is shaded dark gray, for $B\in\R.$ See Figure~\ref{fig:twoParts}. The cell regions are $\reg_A=-1$, $\reg_B=-2$, and $\reg_{A'}=-3$.}
\label{fig:regions}
\end{figure}
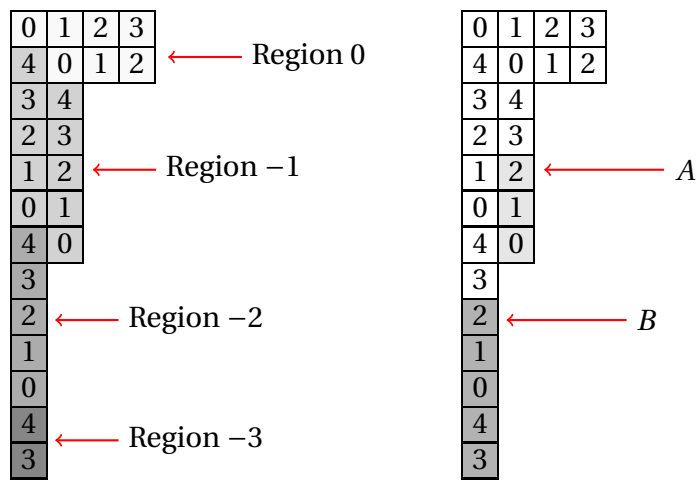

\begin{claim}
  \label{claim:unders}
  There are $\ell$ cells under $A$ ($|\unders{A}|=\ell$) and for all other $B\in\R$, $|\unders{B}|\geq \ell+1$.

\end{claim}
\begin{proof} Suppose that $a_i$ and $a_j$ are as in Definition~\ref{def:abacus}. We may assume that either $0\leq i<j<k$ or $0\leq j-k< i<k$. 
  \omitt{We assume first that $i \bmod k<j\bmod k$ and set $i=i\bmod k$ and
  $j=j\bmod k$. The case $i \bmod k>j\bmod k$ is similar but
  unpleasant.} The condition that there is no $h$ such that $i<h<j$ and
  $a_j\leq a_h\leq a_i$ is key here. Call this condition C, for cover. 

As in Definition~\ref{def:abacus}, we trace the outline of our
partition $\lambda$ to obtain the abacus. The partition $\lambda$ need
not be $\bgspart{n}{k}$, but we do need to start at $(k-1,0)$, where
we imagine the upper left corner of the diagram of $\lambda$ at
$(0,0)$ and each cell being $1$ by $1$.

The step corresponding to the border of $A$ is the first step in the
path of $\lambda$ which is different from the path of $\gamma$. The
first abacus position where the abacus of $\lambda$ is different from
the abacus of $\gamma$ is $a_jk+i$: on $\gamma$'s abacus there is a
bead at this position, on $\lambda$'s, there is not. The left border
of $A$ is at step $a_jk+i$.  By condition C, there are no
beads/horizontal steps at $i+1,\ldots,j-1$ in the row $a_j$ of the
$\lambda$'s abacus, and there is a bead $\ell$ steps further, at
position $a_jk+j$. The bead at position $a_jk+j$ is in row $a_j$,
runner $j$ if $j<k$ and it is in row $a_j+1$, runner $j-k$, if $j\geq
k$. This beads represents the horizontal step in the silhouette of
$\lambda$ which ends $A$'s column.  Therefore,
$|\unders{A}|=i-j=\ell$. There are no more left steps landing on
rod $j$ and by condition C, each step on rod $i$ is followed by
down steps on rods $i+1,\dots, j-1$ until the row after we place
a bead on $i$. Thus we have left steps on $i,\dots,j$ so that
$|\unders{B}|\geq \ell+1$ for $B$ of residue $r_A$ and region
$\reg_A-1,\reg_A-2,\ldots,\reg_{A'}$.
  \end{proof}

 See \cite[Sections 8.2 and 8.3]{LLMS} for more information on
  extended offset sequences \cite[Lemma 9.4]{LLMS} for Bruhat order on
  affine permutations.

\begin{figure}[h]
\begin{tikzpicture}
  \begin{scope}
\def\a{1.7}
\def\b{1.2}
\node (0) at (0,0){$1^{15}$};
\node (1) at (0,\b){$21^{13}$};
\node (2a) at (-\a/2,2*\b){$2^51^{5}$};
\node (2b) at (\a/2,2*\b){$31^{12}$};
\node (3a) at (-\a,3*\b){$2^61^{3}$};
\node (3b) at (0,3*\b){$32^31^{}$};
\node (3c) at (\a,3*\b){$41^{11}$};
\node (4a) at (-\a,4*\b){$32^51^{2}$};
\node (4b) at (0,4*\b){$3^22^21^{5}$};
\node (4c) at (\a,4*\b){$42^21^{7}$};
\node (5a) at (-\a,5*\b){$42^51$};
\node (5b) at (0,5*\b){$3^22^31^3$};
\node (5c) at (\a,5*\b){$4321^6$};
\node (6a) at (-\a,6*\b){$3^5$};
\node (6b) at (0,6*\b){$432^31^2$};
\node (6c) at (\a,6*\b){$4^221^5$};
\node (7a) at (-\a/2,7*\b){$43^22$};
\node (7b) at (\a/2,7*\b){$4^22^21^3$};
\node (8) at (0,8*\b){$4^23^21$};
\node (9) at (0,9*\b){$4^33$};

\draw[thick,blue] (0)--(1);

\draw[thick,blue] (1)--(2a);
\draw[thick,blue] (1)--(2b);

\draw[thick,blue] (2a)--(3a);
\draw[thick,blue] (2a)--(3b);
\draw[thick,blue] (2b)--(3b);
\draw[thick,blue] (2b)--(3c);

\draw[thick,blue] (3a)--(4a);
\draw[thick,blue] (3a)--(4b);
\draw[thick,blue] (3b)--(4a);
\draw[thick,blue] (3b)--(4b);
\draw[thick,blue] (3b)--(4c);
\draw[thick,blue] (3c)--(4c);

\draw[thick,blue] (4a)--(5a);
\draw[thick,blue] (4a)--(5b);
\draw[thick,blue] (4b)--(5b);
\draw[thick,blue] (4b)--(5c);
\draw[thick,blue] (4c)--(5a);
\draw[thick,blue] (4c)--(5c);

\draw[thick,blue] (5a)--(6b);
\draw[thick,blue] (5b)--(6a);
\draw[thick,blue] (5b)--(6b);
\draw[thick,blue] (5b)--(6c);
\draw[thick,blue] (5c)--(6b);
\draw[thick,blue] (5c)--(6c);

\draw[thick,blue] (6a)--(7a);
\draw[thick,blue] (6b)--(7a);
\draw[thick,blue] (6b)--(7b);
\draw[thick,blue] (6c)--(7b);

\draw[thick,blue] (7a)--(8);
\draw[thick,blue] (7b)--(8);

\draw[thick,blue] (8)--(9);
\end{scope}

\end{tikzpicture}
\caption{The poset $\mc{P}(15,5)$}
\label{fig:poset}
\end{figure}

\subsection{Basis of the standard modules}

In Section~\ref{sec:stanMod}, we define the standard modules and in
Section~\ref{sec:stanModBasis} we mention a basis $f_{P,Q}$ for
$\stanmod$. It was originally defined in \cite{Gri2}. Here we will describe it in more detail than in Section~\ref{sec:stanModBasis} and give it a
different indexing set.  Additionally, we use a rescaled version of the polynomials. This affects the action of $\sigma_i$: see \eqref{eq:sigmajack}.

\subsubsection{Generalized Jack polynomials}
\label{sec:genJack}
We'll now use  $$\{(\mu,T)\mid \mu\in\Z^n_{\geq 0} \text{ and $T$ is a standard Young tableau of shape $\lambda$}\}$$
for the indexing set for our basis of $\stanmod(\lambda)$. The bijection between pairs $(\mu,T)$ and the pairs $(P,Q)$ of Section~\ref{sec:stanModBasis} is given by
$$P(b)=w_{\mu}^{-1}(T(b))\quad\text{and}\quad Q(b)=\mu_{P(b)} $$ and
$$\mu_i=Q(P^{-1}(i))\quad\text{and}\quad T(b)=w_{\mu}(b).$$ We view the standard Young tableau $T$ as a function from the boxes of $\lambda$ to $[n]$ and the permutation $w_{\mu}$ is defined by
$$w_{\mu}(i)=|\{1\leq j<i:\mu_j<\mu_i\}|+|\{i\leq j\leq n:\mu_j\leq
\mu_i\}|.$$ The permutation $w_{\mu}$ has the property that
$$\mu_{w^{-1}_\mu(1)}\leq \mu_{w^{-1}_\mu(2)}\leq\ldots\leq
\mu_{w^{-1}_\mu(n)}$$ and it is the longest permutation of $[n]$ with
this property.


Let $\zeroComp$ be the weak composition of $0$ with $n$ parts. The polynomial $\jacko{T}$ is defined to be $v_T$ from Section~\ref{seminormal}. The recursions
\begin{equation}
\jack{(\mu_n+1,\mu_1,\ldots,\mu_{n-1})}{T}=x_ns_{n-1}\cdots s_1\jack{\mu}{T},
\end{equation}
where $\mu=(\mu_1,\ldots,\mu_n)$ and $s_i$ is the simple transposition
$(i,i+1)$,
and
\omitt{
\begin{equation}\jack{s_i\mu}{T}=\sigma_i\jack{\mu}{T}\text{, where $\mu_i<\mu_{i+1}$,} \label{part:recb}
\end{equation}
}

\omitt{$\{\jack{\mu}{T}\}$ is a basis for $\stanmod(\lambda)$ and the action of $\sigma_i$ on it is given by 
}
\begin{equation}\label{eq:sigmajack}
\sigma_i\jack{\mu}{T}=\begin{cases}
\jack{s_i\mu}{T}&\text{if $\mu_i>\mu_{i+1}$}\\
\frac{\delta^2-c^2}{\delta^2}\jack{s_i\mu}{T}&\text{if $\mu_i<\mu_{i+1}$}\\
\jack{\mu}{s_{j-1}T}&\text{if $\mu_i=\mu_{i+1}$, $\ell(s_{j-1}T)>\ell(T)$, and}\\
&\text{$s_{j-1}T$ is a standard Young tableau}\\
1-\left(\frac{1}{\ct{T}{j}-\ct{T}{j-1}}\right)^2\jack{\mu}{s_{j-1}T}&\text{if $\mu_i=\mu_{i+1}$, $\ell(s_{j-1}T)\ell(T)$, and}\\
&\text{$s_{j-1}T$ is a standard Young tableau}\\
0&\text{$s_{j-1}$ is not a standard Young tableau}
\end{cases}
\end{equation}

where $j=w_\mu(i)$ and $\delta=\delta(\mu,T,i)=\mu_i-\mu_{i+1}-c(\ct{T}{w_\mu(i)}-\ct{T}{w_\mu(i)})$,
define $\jack{\mu}{T}$ for all weak compositions.

We can write

\begin{equation}
\jack{\mu}{T}=\sum a_S(c,x_1,\ldots,x_n)\jacko{S}\label{eq:jacksum},
\end{equation}
where $a_S(c,x_1,\ldots,x_n)$ is a polynomial in $x_1,\ldots, x_n$ with coefficients that are rational functions of $c$. See \cite{Gri2} for more on $z_i$ and $\sigma_i$, where $1\leq i\leq n$. 

\subsubsection{Weights}
\label{subsec:weights}
We fix $n$ and $k$ and let $c=1/k$.
The generalized Jack
polynomial $\jack{\mu}{T}$ is an eigenvector for $z_i$: we have
$z_i\jack{\mu}{T}=\wti{\mu}{T}{i}\jack{\mu}{T}$. See \cite{Gri2}. The expression for
$\wti{\mu}{T}{i}$ is $\mu_i+1-\ct{w_{\mu}(i)}{T}c$, where $w_{\mu}$ is defined in Section~\ref{sec:genJack} and the weight
vector $\wt{\mu}{T}$ for $(\mu,T)$ is
$\wt{\mu}{T}=(\wti{\mu}{T}{1},\ldots,\wti{\mu}{T}{n})$.

We will sometimes refer to the weight of a cell. Let $\lambda$ be a
partition of $n$, $B$ a cell in $\lambda$, $S$ a standard Young
tableau of shape $\lambda$, and $\eta$ a weak composition of length
$n$. Suppose $i$ is in cell $B$ of $S$.  Then define $\eta_B$ to be
 $\eta_{w_{\eta}^{-1}(i)}$ and $\wti{\eta}{S}{B}=\eta_B+1-\reg-r/k$, where
$\ctBox{B}=k\reg+r$, $0\leq r <k$.

The index pair $(\mu,T)$ is {\em multiplicity one} if there is not a
different pair $(\mu',S)$ with the same weight vector, where $S$ and
$T$ are standard Young tableaux of the same shape and $\mu$ and $\mu'$
are weak compositions of the same number.

The multiplicity one property is important because we can use it to
prove that the Jack polynomial $\jack{\mu}{T}$ is defined, since there
could be singularities in the coefficients in \eqref{eq:jacksum}.

\begin{claim}\label{claim:multOne}
Suppose that the eigenvalue for a pair $(\mu,T)$  occurs with multiplicity one in a given standard module for a
given parameter $c=c_0$. Then the Jack polynomial $\jack{\mu}{T}$ has no pole at $c_0$.
\end{claim}
\begin{proof} The Jack polynomial $\jack{\mu}{T}$ for generic polynomial parameter $c$ is a sum of terms of the form
$p(c) x^\nu v_U$ where $p(c)$ is a rational function of $c$, which we assume
is written without common divisors between the numerator and
denominator. Suppose $\jack{\mu}{T}$ has a pole at $c=c_0$, and among all such
terms, choose one with the largest power of $(c-c_ 0)$ in the
denominator, say its denominator is divisible by $(c-c_0)^e$ but not by
$(c-c_0)^{e+1}$. Multiplying $\jack{\mu}{T}$ by $(c-c_ 0)^e$ then produces a sum
of terms without pole at $c=c_0$, and so that upon specializing $c$ to $c_0$
the leading term is smaller than that of $\jack{\mu}{T}$. This polynomial is
an eigenfunction for the same eigenvalue as that of $\jack{\mu}{T}$, and
witnesses multiplicity at least two. So if the multiplicity is one,
then $\jack{\mu}{T}$ cannot have a pole at $c=c_0$.
\end{proof}

\subsection{Definition of the map}
\label{subsec:maps}

Fix $n$ and $k$ and let $c=1/k$ and suppose that $\gamma\gtrdot \lambda$ in
$\Pnk$. This section of the paper is concerned with the definition of
the map $\Map$.
We essentially have two maps, a combinatorial one and an algebraic one, both
denoted by $\Map$. 
The combinatorial one takes a tableau of shape $\gamma$ and returns a
composition-tableau pair, where the (weak) composition is of length $n$ and the tableau is of shape $\lambda$. The algebraic one is a map of modules:
$\Map:S^{\gamma}\to\stanmod(\lambda)$, $\Map:\jacko{T}\mapsto
\mapconst{\TDo}\jack{\mu'}{\TIm}$. The combinatorial map defines the algebraic one, up a scalar. In Sections~\ref{BGG} and \ref{sec:multOne} we focus mainly on the combinatorial map, and on algebraic implications in Section~\ref{sec:looseEnds}. 



\omitt{ We will often denote $\jack{\mu}{T}$ simply by its indices $(\mu,T)$,
so that $(\mu,T)$ is both a composition-tableau pair and a generalized
Jack polynomial. Most of our definition and analysis of the map $\Map$
involves only the combinatorial properties of the pair $(\mu,T)$.}

We define the map on the \omitt{basis $\jack{\mu}{\TDo}$}
$T\in\syt(\lambda)$ in several steps. As a note, we can do this
because we need only define the algebraic version on $\jacko{\TDo}$
and extend using \eqref{eq:jacksum}.

Since $\gamma$ covers $\lambda$, the $k$-abacus for $\gamma$ is the same as the
$k$-abacus for $\lambda$, except that $a_i$ and $a_j$ have been
interchanged. Define $m$ as $a_i-a_j$ and define
$\standComp=(\underbrace{m,\ldots,m}_{\textrm{$\ell$ times}},0,\ldots,0)$.

Let $\TDo\in \syt(\gamma)$. $\rev{\TDo}$ is the result of replacing $i$
with $n-i+1$ in $\TDo$. Denote the entries in $\rev{\TDo}$ in the cells from
$\gMinusl$ by $i_1<i_2<\cdots<i_{\ell}$. The weak composition $\mu$
is defined by 
$$\mu_i=
\begin{cases}m&\mbox{if }i\in\{i_1,\ldots,i_{\ell}\}\\
0&\mbox{otherwise},\end{cases}$$
for $1\leq i\leq n.$ Next move the strip $\gMinusl$ to the
positions in $\lMinusg$, so that now we have $P$ of shape
$\lambda$. Calculate $w_{\mu}$ and apply to $P$ to finally obtain
$\TIm$. Then $\Map(\zeroComp,\TDo)=(\mu,\TIm)$ or
$\Map(\jacko{\TDo})=\mapconst{\TDo}\jack{\mu}{\TIm}$, where we discuss
$\mapconst{\TDo}$ in Section~\ref{sec:looseEnds}.  The map
$\phiInv{\eta}{\SIm}$ is defined for any composition $\eta$ of $n$
with $\eta^+=\standComp$ and $\SIm$ a standard Young tableau of shape
$\lambda$ and an element of $\I$ by reversing these steps, even for
$(\eta,\SIm)$ not in the image of $\Map$.

We mention without discussion that the map $\Map$ may be defined on pairs $(P,Q)$ (see Section~\ref{sec:stanModBasis}) instead of pairs $(\zeroComp,\TDo)$ and that it preserves weight. See Claim~\ref{claim:zero} for a special case of this.

We must address several issues. First, if
$\Map(\zeroComp,\TDo)=(\mu,\TIm)$, then we must show $\jack{\mu}{\TIm}$
exists. Second, we must show that $\sigma_i\Map=\Map\sigma_i$, which includes showing that $\mapconst{\TDo}$ is well-defined. We address the first issue immediately in Sections~\ref{BGG} and \ref{sec:multOne}, the second in Section~\ref{sec:looseEnds}.

The function $\jack{\mu}{\TIm}$ may not exist for all $(\mu,\TIm)$ in the
image. The coefficients, which depend on $c$, may not be
defined. What's more, because of the recursion \eqref{eq:sigmajack} which
defines the Jack polynomials, we must consider elements of the
near-image:

\begin{definition}
Suppose $\gamma\gtrdot\lambda$ The pair $(\mu,\TIm)$, which represents
an element of $\stanmod(\lambda)$, is in the {\em near-image of
$\Map$} if it is not in the image but there is an $i$, $1\leq i\leq
n-1$, such that $(s_i\mu,\TIm)$ is in the image of $\Map$.
\end{definition}

Claim~\ref{claim:multOne} shows that the polynomial $\jack{\mu}{\TIm}$
is well-defined when $(\mu,\TIm)$ is multiplicity
one. Section~\ref{sec:multOne} is devoted to proving the
$\jack{\mu}{\TIm}$ in the image and near-image of $\Map$ are
multiplicity one.

\begin{example}
\label{ex:mapExample} In this example, $n$ is $11$, $k$ is $5$, $\gamma$ is $(3,2,2,2,1,1)$ and $\lambda$ is $(3,3,3,2)$. In Figure~\ref{fig:mapExample}, we map $(\zeroComp,\TDo)$ to $(\mu^1,\TIm)$. Let $\mu^2=(0,1,0,0,1,0,0,0,0,0,0)$. Then $\phi^{-1}(\mu^2,\TIm)$ is in the near-image, since it is not in the image and $(s_4\mu^2,\TIm)$ is.
\end{example}

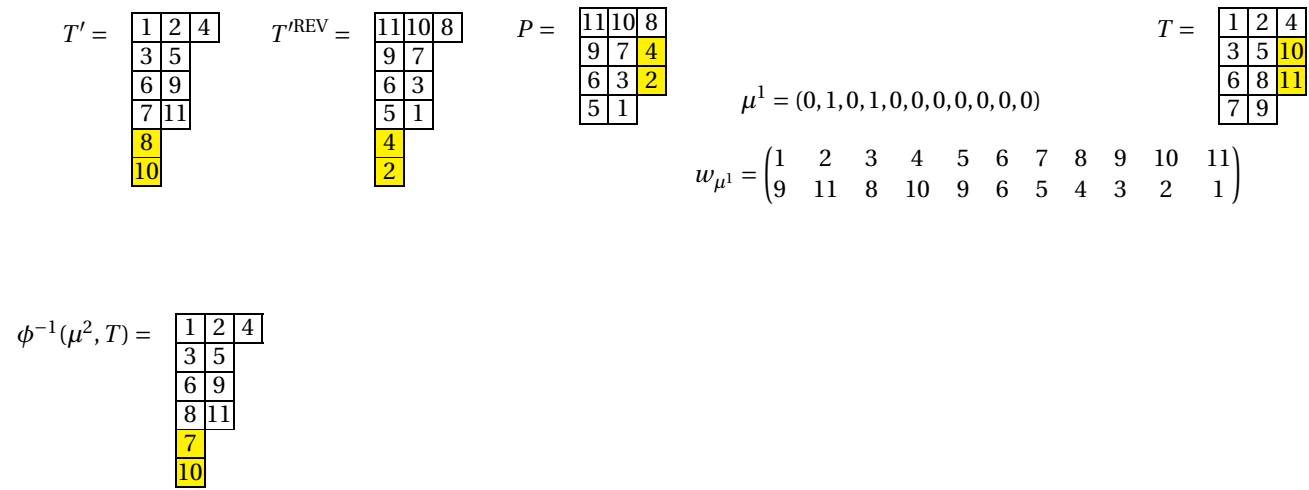
\begin{figure}[h]
   \ytableausetup{boxsize=1.2em}
\begin{tikzpicture}[font=\footnotesize]
  \begin{scope}

\begin{scope}[shift={(0,0)}]
\def\a{.45}
\node at (0,0) {$T'=$\quad\begin{ytableau}1&2&4\\3&5\\6&9\\7&11\\*(yellow)8\\*(yellow)10\end{ytableau}};
\node at (3,0) {$\rev{T'}=$\quad\begin{ytableau}11&10&8\\9&7\\6&3\\5&1\\*(yellow)4\\*(yellow)2\end{ytableau}};
\node at (6,\a) {$P=$\quad\begin{ytableau}11&10&8\\9&7&*(yellow)4\\6&3&*(yellow)2\\5&1\end{ytableau}};
\node at (10,0) {$\mu^1=(0,1,0,1,0,0,0,0,0,0,0)$};
\node at (11,-1){$w_{\mu^1}=\begin{pmatrix}1&2&3&4&5&6&7&8&9&10&11\\9&11&8&10&9&6&5&4&3&2&1\end{pmatrix}$};
\node at (14.5,\a){$T=$\quad\begin{ytableau}1&2&4\\3&5&*(yellow)10\\6&8&*(yellow)11\\7&9\end{ytableau}};
\end{scope}

\begin{scope}[shift={(0,-4)}]
\node at (0,0){$\phi^{-1}(\mu^2,T)=$\quad\begin{ytableau}1&2&4\\3&5\\6&9\\8&11\\*(yellow)7\\*(yellow)10\end{ytableau}};
\end{scope}

\end{scope}

\end{tikzpicture}
\caption{Figures for Example~\ref{ex:mapExample}}
\label{fig:mapExample} 
\end{figure}

\omitt{I claim that if $\gamma\gtrdot\lambda$ in an edge in $\Pnk$, then the
image of $\stanmod(\gamma)$ is all multiplicity one. }

\section{Multiplicity one}
\label{sec:multOne}
This section is dedicated to proving that following theorem.

\begin{theorem}
\label{thm:multOne}
  If $\gamma\gtrdot\lambda$ in an edge in $\Pnk$, then the
image and near-image of $\stanmod(\gamma)$ under $\Map$ is all multiplicity one. 

\end{theorem}

\omitt{
We begin by introducing a lot of notation and collecting many simple
facts which can be deduced about $\gamma$ and $\lambda$ and the image
of $\stanmod(\gamma)$ in $\stanmod(\lambda)$ under $\phi$ in the case
when $\gamma\gtrdot\lambda$. The residues and regions, defined below,
are constrained by the cover relations in the poset; we state the
properties we'll need. Next, in Section~\ref{subsec:multiset} we show
that the multiset of weights is unique for a pair composed of a
composition and tableau of shape $\lambda$, under certain
conditions. Finally, we prove the theorem in \ref{subsec:multOne}.}
\subsection{The image and the near-image}

\begin{notation}
From now on, fix $n$ and $k$ and fix $\gamma\gtrdot\lambda$, both in
$\Pnk$, with extended $k$-abacuses as in Definitions~\ref{def:abacus}, \ref{def:extAba}, and \ref{def:cover}. \omitt{Denote $a_i$ by $a$ and
$a_j$ by $b$.} Let $\ell=j-i$ and $m=a_i-a_j$. \omitt{As in
Section~\ref{sec:partialOrder}, let $A$ denote the cell at the top of
the strip $\lMinusg$ in $\lambda$ and $A'$ the cell at the top of the
strip $\gMinusl$ in $\gamma$. }We fix $(\mu,\TIm)$ the rest of the
section. $(\mu,\TIm)$ will denote a pair in the image or near-image of
the map $\Map$ from $\stanmod(\gamma)$ to $\stanmod(\lambda)$.  Let
$\I=\I(\gamma,\lambda)$ denote the set of $\SIm\in \syt(\lambda)$ which
have entries $n,n-1,\ldots,n-\ell+1$ in
$\lMinusg$. Claim~\ref{claim:muTType}\eqref{claim:IPart} shows why
this set is important.   Let $\col$ be the index of the
column in $\gamma$ that contains $\gMinusl$, and let
$r_0,\ldots,r_0+\ell-1$ be the indices of the rows that contain
$\gMinusl$. 

\omitt{A cell $(x,y)$ in $\lambda$ is in region $\reg$ with residue $r$ if
$y-x=k\reg +r$, where $0\leq r<k$. A region which we number $\reg$
would be numbered $\reg-1$ in \cite{GKS}.  In general, for a cell $B$
we denote its region by $\reg_B$ and residue by $r_B$. We have
$\reg_A$ and $r_A$ for the region and residue of $A$ and $\reg_{A'}$
and $r_{A'}$ for $A'$. Additionally, we define the diagonal $D_B$ of a
cell $B$: $$D_B=\{C\in\lambda|\ctBox{C}=\ctBox{B}\}.$$}
\end{notation}

\omitt{
Claims~\ref{claim:muTType}, and
\ref{claim:mAndReg} follow from the definitions of the cover relation,
the abacus, and the map $\phi$. Additionally, note that
Claims~\ref{claim:singleCol}, \ref{claim:mAndReg}, and
\ref{claim:unders} analyze $\gamma$ and $\lambda$ and don't address
tableaux.
}

We can pinpoint the inverse image of a pair $(\eta,\SIm)$ in the near
  image.  We will need the following claim for
  Claim~\ref{claim:muTType} and Proposition~\ref{prop:sinuOK}.

  \begin{claim} \label{claim:smallProblem} Suppose $(\eta,\SIm)$ is in
    the near-image: $\phiInv{\eta}{\SIm}=(\zeroComp,\SDo)$, where $\SDo$ is
    not a standard Young tableau, and
    $\phiInv{s_i\eta}{\SIm}=(\zeroComp,\SDo_1)$, where $\SDo_1$ is. Also
    suppose $\SIm\in\I$. Then $\eta_{i+1}=m$, $\eta_i=0$, and $\SDo$ has
    exactly one violation. Either there is an $h$, $0\leq h< \ell$,
    such that $\SDo_{r_0+h,\col}=n-i-1$ and $\SDo_{r_0+h,\col-1}=n-i$ or
    $\SDo_{r_0,\col}=n-i-1$ and $\SDo_{r_0-1,\col}=n-i$.
\end{claim}

\begin{proof}
It is straightforward to calculate that
$w_{s_i\eta}=w_{\eta}s_i$. There is a violation when
$w_0w_{\eta}^{-1}$ is used to invert $\Map$, but not when
$w_0s_iw_{\eta}^{-1}$ is. That is, we only need to swap $n-i$ and
$n-i-1$ in $\SDo$ to remove its problem. Working backwards, there is a
cell $\lambda\cap\gamma$ with label $a$ in $\SIm$ and a cell in
$\lMinusg$ with label $b$ in $\SIm$ such that $w_{\eta}^{-1}(a)=i$ and
$w_{\eta}^{-1}(b)=i+1$. Since $\SIm\in\I$, $b>n-\ell$ and
$w_{\eta}^{-1}(b)\in\{i_1,\ldots,i_{\ell}\}$;
i.e. $\eta_{i+1}=m$. Similarly, $\eta_i=0$.
\end{proof}

\begin{claim}
  \label{claim:muTType}
  For $(\mu,\TIm)$  in the image or near-image, we have 
\begin{enumerate}
\item \label{claim:muPlusPart}$\mu^+=(\underbrace{m,\ldots,m}_\text{$\ell$ times},0,\ldots,0),$ 
\item \label{claim:IPart} the cells in $\TIm$ in $\lMinusg$ are filled with $\{n-\ell+1,\ldots,n\}$ (see Figure~\ref{fig:I}), and   
\item \label{claim:muBPart}$\mu_B=m$ if and only if $B\in\lMinusg$.
\end{enumerate}
\end{claim}
  \ytableausetup{mathmode, boxsize=2.5em}
  \begin{figure}[ht]
  \begin{ytableau}
    \scriptstyle n-\ell+1\\
    \scriptstyle n-\ell+2\\
    \none[\vdots]\\
    n
  \end{ytableau}
\caption{The cells in $\lMinusg$ for $S\in\I$} 
\label{fig:I}
  \end{figure}

  \begin{proof} First consider $(\mu,\TIm)$ in the
image. Part \eqref{claim:muPlusPart} is a direct consequence
of the definition of $\Map$. As in the definition of $\Map$, denote the entries
in $\rev{\TDo}$ in the cells from $\gMinusl$ by
$i_1<i_2<\cdots<i_{\ell}$, where $i_{\ell}$ will be in the top-most
cell of $\gMinusl$ in $\rev{\TDo}$ and also of $\lMinusg$ in $P$. By the
definition of $w_{\mu}$, $w_{\mu}(i_j)=n-j+1$. When we apply $w_{\mu}$
to $P$, we have a tableau $\TIm$ of the form claimed
by \eqref{claim:IPart}. Finally, suppose the entry in $B$ is $j$. Then
$\mu_B=m$ if and only if the
$w_{\mu}^{-1}(j)\in\{i_1,\ldots,i_{\ell}\}$ if and only if
$n-\ell+1\leq j\leq n$ if and only if $B\in\lMinusg$.

Now suppose $(\mu,\TIm)$ is in the near-image. Since $(s_i\mu,\TIm)$ is in
the image, by the first part of this proof, \eqref{claim:muPlusPart}
and \eqref{claim:IPart} hold. For \eqref{claim:muBPart}, in light of
Claim~\ref{claim:smallProblem}, we only need to consider the cells
containing $j_1$ and $j_2$, where $w_{s_i\mu}^{-1}(j_1)=i$ and
$w_{s_i\mu}^{-1}(j_2)=i+1$. We have $(s_i\mu)_{i+1}=m$ and
$(s_i\mu)_i=0$, because $(s_i\mu,\TIm)$ is in the image. This shows that
$j_1$ is the entry of a cell in $\lambda\cap\gamma$ and $j_2$ is in a
cell in $\lMinusg$. Since $w_{s_i\mu}=w_{\mu}s_i$, we know that
$w_{\mu}^{-1}(j_1)=i+1$ and $w_{\mu}^{-1}(j_2)=i$ and we are done.
\end{proof}

  \omitt{What's more (I think), every tableau which is filled in this manner will appear in the image at least once, with different permutations of $(m,m\ldots,m,0\ldots,0)$.}

\subsection{The multiset of weights}
\label{subsec:multiset}
Throughout this section, $n$ and $k$ are fixed and
$\gamma\gtrdot\lambda$ is an edge in $\Pnk$. We continue to use $(\mu,\TIm)$ for a pair in the image or near-image and $(\eta,\SIm)$ represents an arbitrary element of the basis for $\stanmod(\lambda)$.

\begin{notation}
For a standard Young tableau $U$ of any shape and a weak composition
$\eta$, let $\M{\eta}{U}$ denote the multiset of weights of $(\eta,U)$
and let $\mult{x}{\eta}{U}$ be the multiplicity of $x$ in
$\M{\eta}{U}$. Further, for any partition $\delta$, let
$\multB{x}{\delta}$ be the number of cells $B$ in the diagram of
$\delta$ such that $1-\ct{B}{\delta}/k$ is equal to $x$. In other
words, $\multB{x}{\delta}$ is $\mult{x}{\zeroComp}{T}$ where $T$ is a
tableau of shape $\delta$.  See Claim~\ref{claim:singleCol} for a
reminder on regions and the notation $\R'$.\omitt{Recall that $\R'$ is
the cells in $\lambda$ which have residue $r_A$ and $\R$ the subset of
$\R'$ which lie in region $\reg$, for
$\reg_{A'}\leq \reg\leq \reg_A$.} 
\end{notation}

\begin{claim}
\label{claim:sameMulti}
  Let $(\eta^1,\SIm_1)$ and $(\eta^2,\SIm_2)$ be pairs where $(\eta^1)^+=(\eta^2)^+$ and $\SIm_1,\SIm_2\in \syt(\lambda)$. Let $\ell_1$ be the number of nonzero parts of $\eta^1$. Suppose the positions of $n, n-1,\ldots,n-\ell_1+1$ are the same in 
 $\SIm_1$ as in $\SIm_2$. Then the multiset of weights of $(\eta^1,\SIm_1)$ is
  the same as the multiset of weights of $(\eta^2,\SIm_2)$.
  \end{claim}

\begin{proof}
Suppose $i_1<\cdots<i_{\ell_1}$ and $j_1<\cdots<j_{\ell_1}$ are the
positions in $\eta^1$ and $\eta^2$ respectively of the nonzero
entries. Then $w_{\eta^1}(i_x)=n-x+1$ and $w_{\eta^2}(j_x)=n-x+1$ for
$1\leq x\leq \ell_1$. Therefore
$\wti{\eta^1}{\SIm_1}{i_x}= \wti{\eta^2}{\SIm_2}{j_x}=m+1-\ct{n-x+1}{\SIm_1}c$. The
situation is the same for the positions of the zero enties: since the
set of remaining cells in both $\SIm_1$ and $\SIm_2$ is the same, the multiset of cell contents will be the same.
\end{proof}

We will only be concerned about the multiset of weights in this
sub-section, so by Claim~\ref{claim:sameMulti} we assume
$\mu=\standComp=(m,m\ldots,m,0\ldots,0)$.

\begin{claim}
  \label{claim:etaRegion}
  Suppose $\M{\eta}{\SIm}=\M{\mu}{\TIm}$ for some $\eta$ and $\SIm$. Let $B$ be
  a cell in $\lambda$. If $\eta_B\neq 0$, then $\reg_B$, the
  region of $B$, satisfies $\reg_{A'}\leq \reg_B\leq \reg_A$.
  \end{claim}

\begin{proof}
Recall that for $\alpha\in\Z^n$, $\TIm$ a standard Young tableau of shape
$\lambda\vdash n$, and $B\in\lambda$, $\alpha_B$ is
$\alpha_{w_{\alpha}^{-1}(i)}$, where $i$ is the entry in cell $B$ of
$\TIm$.  Let $\reg$ be the largest region number of a region which
contains a cell $B$ such that $\eta_B> 0$. Fix a cell $B$ with
$\eta_B> 0$ in region $\reg$. Consider the diagonal $D_B$. Now
suppose, for contradiction, that $\reg>\reg_A$. Since $\reg>\reg_A$,
we have $\mu_{B'}=0$ by Claim~\ref{claim:muTType}\eqref{claim:muBPart}
and $\wti{\mu}{\TIm}{B'}=1-\reg-r/k$ for all $B'\in D_B$. A cell with a
higher region number than $\reg$ will have a weight that is strictly
smaller than a cell in $D_B$; likewise, a cell with a lower region
number will have a higher weight.  That means that
$\mult{1-\reg-r/k}{\mu}{\TIm}=|D_B|$. Therefore, there should be $|D_B|$
cells $B'$ in the Young diagram of $\lambda$ such that
$\wti{\eta}{\SIm}{B'}=1-\reg-r/k$. These cells cannot all be in $D_B$,
since there is the cell $B\in D_B$ with
$\wti{\eta}{\SIm}{B}=\eta_B+1-\reg-r/k>1-\reg-r/k$. Let $C$ be a cell in
$\lambda$ with the same residue. If $\reg_C>\reg$, then $\eta_C=0$,
forcing $\wti{\eta}{\SIm}{C}=1-\reg_C-r/k<1-\reg-r/k$. If $\reg_C<\reg$,
then $\eta_C+1-\reg_C-r/k=1-\reg-r/k$ means that $\eta_C<0$. Therefore
$\reg\leq
\reg_A$.

Of the regions which contain a cell $B$ with $\eta_B> 0$, pick the
region with the smallest region number. Let $\reg=\reg_{B_0}$ be the number
of this region and let $B_0$ be a cell in it such that $\eta_{B_0}>
0$. Suppose $\reg_{B_0}<\reg_{A'}$, for contradiction. Since
$\M{\eta}{\SIm}=\M{\mu}{\TIm}$, there is a cell $B_1$ with the same residue
as $B_0$ such that $\mu_{B_1}-\reg_{B_1}=\eta_{B_0}-\reg_{B_0}$.

If $\reg_{B_1}<\reg_{B_0}$, then $\eta_B=\mu_B=0$ for all $B\in
D_{B_1}$. Therefore, each cell in $D_{B_1}$ is responsible for the same
weight in $\M{\eta}{\SIm}$ as it is in $\M{\mu}{\TIm}$, so the cells in
$D_{B_1}$ ``cancel'' themselves out. There must be a cell $B_1$ with
$\reg_{B_1}\geq \reg_{B_0}$ with the same residue as $B_0$ and
$\mu_{B_1}-\reg_{B_1}=\eta_{B_0}-\reg_{B_0}$.

We assume $\reg_{B_1}\geq \reg_{B_0}$ and
$\mu_{B_1}-\reg_{B_1}=\eta_{B_0}-\reg_{B_0}$. Thus
$\mu_{B_1}-\eta_{B_0}=\reg_{B_1}-\reg_{B_0}\geq 0$. We have the following
expression for $\mu_B$:
$$\mu_B=\begin{cases}m&\text{if $B\in
  \lMinusg$}\\0&\text{otherwise,}\end{cases}$$ and may conclude
$B_1\in\lambda\setminus\gamma$ and $\reg_{B_1}\leq\reg_A\leq\reg_{B_1}+1$. Therefore,
$m=\reg_A-\reg_{A'}<\reg_A-\reg_{B_0}\leq\reg_{B_1}+1-\reg_{B_0}$, so that
$\reg_{B_1}-\reg_{B_0}\geq m.$

On the other hand, $\mu_{B_1}-\eta_{B_0}=\reg_{B_1}-\reg_{B_0}$ means that
$\reg_{B_1}-\reg_{B_0}=m-\eta_{B_0}<m.$ Therefore $\reg\geq \reg_{A'}$.

  \end{proof}

\begin{claim} \label{claim:etaSum} Suppose that $\M{\eta}{\SIm}=\M{\mu}{\TIm}.$ Then we have
$$\sum_{B\in \R'}\eta_B=m.$$
  \end{claim}

By Claim~\ref{claim:etaRegion}, Claim~\ref{claim:etaSum} implies $\sum_{B\in \R'}\eta_B=m.$ 
\begin{proof}
  \begin{eqnarray*}
    \sum_{B\in\R'}\left(\eta_B+1-\ct{B}{\lambda}/k\right)&=&\sum_{B\in\R'}\left(\mu_B+1-\ct{B}{\lambda}/k\right)\\
    \sum_{B\in\R'}\eta_B+\sum_{B\in\R'}(1-\ct{B}{\lambda}/k)&=&m+\sum_{B\in\R'}(1-\ct{B}{\lambda}/k)\\
    \sum_{B\in\R'}\eta_B&=&m\\
    \end{eqnarray*}
  \end{proof}

\begin{proposition}
  \label{prop:difftMultisets}
  Let $(\eta,\SIm)$ be such that $\eta\vdash\ell m$ and $S\in
  \syt(\lambda)$. If either $\eta^+\neq\mu$ or the positions of
  $n,n-1,\ldots,n-\ell+1$ are not the same in $\SIm$ as in $\TIm$, then the
  multiset of weights of $(\eta,\SIm)$ is not equal to the multiset of
  weights of $(\mu,\TIm)$
  \end{proposition}

\begin{proof}

Assume, for contradiction, we have $\eta^+\neq\mu$ but 
$\M{\eta}{\SIm}$ and $\M{\mu}{\TIm}$ are the same. By Claim~\ref{claim:sameMulti}, we may assume
$\eta=\eta^+$. We consider $\eta=\mu$ but $\SIm\not\in\I$ as a second case.

Here we need the bijection given in Section~\ref{sec:genJack} to produce the pair $(P,Q)=(P(\eta,\SIm),Q(\eta,S))$. The pair $(P,Q)$ are described in Section~\ref{sec:stanModBasis}; $P$ also appears in description of the map $\Map$.  
We use $Q=Q(\eta,\SIm)$ to see that $\eta_A<m$. If $\eta_A$ were at least
$m$, then since the labels in $Q$ in $\unders{A}$ are weakly
increasing, we'd have $|\eta|\geq \eta_A|\unders{A}|=\eta_A\ell\geq
m\ell$, so that $\eta_A$ can be at most $m$. But if $\eta_A=m$ and for
all $B\in\unders{A}$, $\eta_B=m$, we'd have $\eta=\mu$. Therefore, we
can assume that $\eta_A<m$ and  $\sum_{B\neq A}\eta_B>0$. We have a contradiction:

  \begin{eqnarray*}
    m\ell=\sum_i\eta_i&\geq&\sum_{B\in\R}\sum_{B'\in\unders{B}}\eta_{B'}\\
    &\geq&\sum_{B\in\R}\eta_B|\unders{B}|\\
    &\geq&\eta_A\ell+\sum_{B\in\R,B\neq A}\eta_B(\ell+1)\text{ by Claim~\ref{claim:unders}}\\
    &\geq&\sum_{B\in\R}\eta_B\ell+\sum_{B\in\R,B\neq A}\eta_B\\
    &=&m\ell+\sum_{B\neq A}\eta_B\text{ by Claim~\ref{claim:etaSum}}\\
    &>&m\ell.
  \end{eqnarray*}

We rely on Claim~\ref{claim:etaRegion} to restrict the sum to $\R$, we consider the weakly increasing labels in $Q(\eta,\SIm)$ again to replace $\sum_{B'\in\unders{B}}\eta_{B'}$ by $\eta_B|\unders{B}|$, and use Claim~\ref{claim:unders} several times. 
  
  We have proved the case $\eta^+\neq\mu$.

  What if $\eta^+=\mu$, but $\SIm\not\in\I$? Then $\eta_A\neq m$ and
  since $\eta^+=\mu$, we have $\eta_A=0$. There must be $B\in\R$ such
  that $\eta_B\neq 0$. But then by Claim~\ref{claim:unders} we
  have $$m\ell=\sum_i\eta_i\geq\eta_B|\unders{B}|\geq m(\ell+1).$$ In
  this case also, we have a contradiction and must have a different
  multiset of weights.

\end{proof}
\subsection{Multiplicity one}
\label{subsec:multOne}

Suppose $(\eta,\SIm)$ and $(\nu,\TIm)$ are such that $\eta,\nu\vdash\ell m$,
$\SIm,\TIm\in \syt(\lambda)$, and either $\eta\neq\nu$ or $\SIm\neq \TIm$.  By
Proposition~\ref{prop:difftMultisets}, we only need to show that even
if $\eta^+=\nu^+$ and $n,n-1,\ldots,n-\ell+1$ have the same positions
in $\SIm$ and $\TIm$, we still have $\wt{\eta}{\SIm}\neq\wt{\nu}{\TIm}$.

\begin{claim}
\label{claim:phiInv}  
Let $\eta$ be a composition of $n$ with $\eta^+=\standComp$ and
$\SIm\in\syt(\lambda)$ an element of
$\I$. Then $\phiInv{\eta}{\SIm}=(\zeroComp,\SDo)$,  where $\SDo$ is a tableau of shape
$\gamma$. $\SDo$ violates the standard Young tableau conditions at most
on the boundary of the column strip $\gMinusl$. More precisely,  $\SDo_{ij}<\SDo_{i+1,j}$ and $\SDo_{ij}<\SDo_{i,j+1}$ except
possibly $\SDo_{i,\col-1}>\SDo_{i,\col}$ for
$i\in\{r_0,\ldots,r_0+\ell-1\}$ and possibly
$\SDo_{r_0-1,\col}>\SDo_{r_0,\col}$. If $\SDo_{ij}<\SDo_{i+1,j}$ and
$\SDo_{ij}<\SDo_{i,j+1}$ all $i,j$, then $(\eta,\SIm)$ is in the image of
$\Map$.
\end{claim}

\begin{proof} 
  To calculate $\phiInv{\eta}{S}$, we apply $w_0w_{\eta}^{-1}$ to $S$,
  then move the entries in $\lMinusg$ to the cells in
  $\gMinusl$. We must show that we can't have a violation within
  the $\gamma\cap\lambda$ cells of $\SDo$ or within the cells
  $\gMinusl$ of $\SDo$.

  Let $i_1<i_2<\ldots<i_{\ell}$ be the coordinates in $\eta$ such that
  $\eta_{i_j}=m$ and $j_1<\ldots<j_{n-\ell}$ be
  $[n]-\{i_1,\ldots,i_{\ell}\}$. Then

  $$w_0w^{-1}_{\eta}(h)=\begin{cases}n-j_{n-h+1}+1&\text{if $h\leq n-\ell$}\\
  n-i_{n-h+1}+1&\text{if $h> n-\ell$.}\end{cases} $$

  Since $S\in\I$, all entries in $\gamma\cap\lambda$ of $S$ are less
  than $n-\ell$. By the expression for $w_0w_{\eta}^{-1}$, $a<b$ in
  $\gamma\cap\lambda$ of $S$ means
  $w_0w_{\eta}^{-1}(a)<w_0w_{\eta}^{-1}(b)$ in $\SDo$ and there is no
  violation in $\gamma\cap\lambda$. Similarly, all entries in
  $\lMinusg$ in $S$ are at least $n-\ell+1$ and there is no
  violation in $\gMinusl$ in $\SDo$.

\end{proof}
  
Note that we have shown that in $\SDo$, where
$(\zeroComp,\SDo)=\phiInv{\eta}{S}$, we have the following situation,
where $i_1<i_2<\ldots<i_{\ell}$ are the coordinates in $\eta$ such
that $\eta_{i_j}=m$ and $S\in\I$.

  \ytableausetup{mathmode, boxsize=3.5em}
  \begin{figure}[ht]
  \begin{ytableau}
    \scriptstyle n-i_{\ell}+1\\
    \scriptstyle n-i_{\ell-1}+1\\
    \none[\vdots]\\
    \scriptstyle n-i_1+1
  \end{ytableau}
  \caption{The cells in $\gMinusl$ for $\SDo=\phiInv{\eta}{\SIm}.$}
  \label{fig:invType}
  \end{figure}

\begin{claim}
  \label{claim:zero} Let $\eta^+=\standComp$, let $\SIm\in\syt(\lambda)$
    be an element of $\I$, and let $(\zeroComp,\SDo)=\phiInv{\eta}{\SIm}$. $\SDo$ need
    not be a an element of $\syt(\gamma)$. Then
    $\wt{\eta}{\SIm}=\wt{\zeroComp}{\SDo}$, where $\zeroComp$ is the composition of length
    $n$ with all 0 parts.
\end{claim}
\begin{proof}
The entry in $B\in\lMinusg$ is $n-\ell+x$ for some $x$ such that
$1\leq x\leq \ell$. It contributes the entry in coordinate
$w_{\eta}^{-1}(n-\ell+x)=i_{\ell-x+1}$ of the weight vector
$\wt{\eta}{\SIm}$. We have
$\wti{i_{\ell-x+1}}{\eta}{\SIm}=m+1-\ctBox{B}/k$. Say that $B$ is moved
to the cell $B'\in\gMinusl$ under $\Map^{-1}$. The entry in $B'$
is $w_0w_{\eta}^{-1}(n-\ell+x)=n-i_{\ell-x+1}+1$ and it contributes
the entry in coordinate $w_0(n-i_{\ell-x+1}+1)=i_{\ell-x+1}$ of the
weight vector $\wt{0}{\SDo}$. Since $\ctBox{B'}=\ctBox{B}-m$, the
entries are the same.

We have a similar situation for $B\in\lambda\cap\gamma$. In $\SIm$, we
have the entry in $B$ is $x$, where $1\leq x \leq n-\ell$. Then
$w_{\eta}^{-1}(x)$ is $j_{n-\ell-x+1}$, so that
$\wti{j_{n-\ell-x+1}}{\eta}{S}$ is $0+1-\ctBox{B}$. In $\SIm$, the
label in $B$ is $w_0w_{\eta}^{-1}(x)=n-j_{n-\ell-x+1}+1$, so that
$\wti{B}{\zeroComp}{\SDo}=\wti{w_0^{-1}(n-j_{n-\ell-x+1}+1)}{\zeroComp}{S\\SDo}=\wti{j_{n-\ell+1}}{\zeroComp}{\SDo}$
is still $0+1-\ctBox{B}$.
  
  \end{proof}
  \begin{claim}
    \label{claim:diffInv}
Let $(\eta,\SIm)$ and $(\mu,\TIm)$ be such that $\eta^+=\mu^+=\standComp$ and
both $\TIm$ and $\SIm$ are elements of $\I$.  Let $(\zeroComp,\SDo)=\phiInv{\eta}{\SIm}$ and
$\TDo=\phiInv{\mu}{\TIm}$.

 If $\eta\neq\mu$ or $\TIm\neq \SIm$, then $\TDo\neq \SDo$. 
\end{claim}
  \begin{proof}
    If $\eta\neq\mu$, then the entries $\gMinusl$ in $\SDo$ are
    different from those in $\TDo$. See Figure~\ref{fig:invType}. If
    $\eta=\mu$, but $\SIm\neq \TIm$, then the difference must occur in
    $\lambda\cap\gamma$. Say we have a box $B\in\lambda\cap\gamma$
    whose entry in $\SIm$ is $x$ and whose entry in $\TIm$ is $y\neq
    x$. Then its entry in $\SDo$ is $w_0w_{\mu}^{-1}(x)$ and in $\TDo$
    it is $w_0w_{\mu}^{-1}(y)\neq w_0w_{\mu}^{-1}(x)$.
    \end{proof}

  \begin{proposition}
    \label{prop:bigPropOne}
Let $(\eta,\SIm)$ and $(\mu,\TIm)$ be such that
$\eta^+=\mu^+=\standComp$ and both $\TIm$ and $\SIm$ are elements of $\I$. Assume
$(\mu,\TIm)$ is in the image of $\Map$. Let $(\zeroComp,\SDo)=\phiInv{\eta}{\SIm}$
and $(\zeroComp,\TDo)=\phiInv{\mu}{\TIm}$. $\TDo\in \syt(\gamma)$, whereas $\SDo$ may not be.  Then, if $\eta\neq\mu$ or $\TIm\neq \SIm$, then $\wt{\eta}{\SIm}\neq\wt{\mu}{\TIm}$.
\end{proposition}

\begin{proof}

As in Claim~\ref{claim:phiInv}, let $\col$ be
the index of the column in $\gamma$ that contains $\gMinusl$,
and let $r_0,\ldots,r_0+\ell-1$ be the indices of the rows that
contain $\gMinusl$.

By Claims~\ref{claim:zero} and \ref{claim:diffInv}, it is enough to
show that $\wt{\zeroComp}{\SDo}\neq\wt{\zeroComp}{\TDo}$. Assume, for contradiction, that
$\SDo\neq \TDo$ and $\wt{0}{\SDo}=\wt{0}{\TDo}$. Two facts are clear but
important to note. The first is that the entries in a diagonal of
standard Young tableaux are strictly increasing reading from the
top-left to the bottom-right. The second is that since
$\wt{0}{\SDo}=\wt{0}{\TDo}$, in each diagonal of $\gamma$, the set of
entries of that diagonal in $\SDo$ is that same as the set of entries
from the same diagonal in $\TDo$, although the order in the diagonal may
be different. We may therefore assume that $\SDo$ is not a standard
Young tableau and the entries in at least one of the diagonals are a
nontrivial permutation of the entries of the same diagonal in $\TDo$. A
{\em violation} on a diagonal occurs when an entry is larger than an
entry to its south-east, both on the same diagonal. A violation on a
diagonal forces either a row or column violation.

 By
Claim~\ref{claim:phiInv}, the only violations in $\SDo$ are on the
boundary of $\gamma\cap\lambda$ and
$\gamma\setminus\lambda$. Diagonals of $\SDo$ which don't have their
rightmost cell in $\gamma\setminus\lambda$ can have no
violations. Consider diagonal $q$ which
ends in $\gamma\setminus\lambda$. For ease of notation, we assume $q\geq 0$-- the arguments are the same for $q\in\Z$. When we denote the entries of
diagonal $q$ in $\SDo$ by $\{\SDo_{1,1+q},\SDo_{2,
  2+q},\ldots,\SDo_{\col-q,\col}\}$, we have $$\SDo_{1,1+q}<\SDo_{2,
  2+q}<\ldots<\SDo_{\col-1-q,\col-1}$$ and no information about how
$\SDo_{\col-1-q,\col-1}$ and $\SDo_{\col-q,\col}$ are ordered. If $\col=1$,
then diagonal $q$ in $\SDo$ has the same entries as diagonal $q$ in
$\TDo$, so assume $\col>1$. Suppose diagonal $q$ has a violation. There
exists an integer $k=k(q)$, $1\leq k <\col-q$ such that

\begin{equation}
\label{eq:diagperms}
\SDo_{i,i+q}=\begin{cases}\TDo_{i,i+q}&\text{if $1\leq i<k$}\\
\TDo_{i+1,i+1+q}&\text{if $k\leq i<\col-q$}\\
\TDo_{k,k+q}&\text{if $i=\col-q$}\\

\end{cases}
\end{equation}

Since $\TDo_{k+1,k+1+q}>\TDo_{k,k+(q+1)}$, we must also have a violation
in diagonal $(q-1)$ with $k(q-1)=k(q)+1$. That is, if $\TDo_{k,k+q}$ is
cycled to the end of diagonal $q$ in $\SDo$, then $\TDo_{k+1,k+1+q}$ must
be cycled to the end of diagonal $q+1$ in $\SDo$. Otherwise, we'd have
$\TDo_{k+1,k+1+q}$ above $\TDo_{k+1,k+q}$, which would be violation in a
column inside $\gamma\cap\lambda$. See
Figure~\ref{fig:TPrimeSPrime}. The cells cycled to
$\gamma\setminus\lambda$ are all from the same column in
$\gamma\cap\lambda$; call it column $\cola$.

Now we obtain our contradiction. We want to show there must be a
violation in column $\cola$ of $\SDo$.  Suppose $q_0$ is the maximum index
of a diagonal in $\SDo$ with a violation. We must be able to cycle all
entries in column $\cola$ below diagonal $q_0$ to
$\gamma\setminus\lambda$. However, since $\cola<\col$, there are more
entries in column $\cola$ below diagonal $q_0$ than there entries in
column $\col$ below diagonal $q_0$, so this is impossible.

\end{proof}

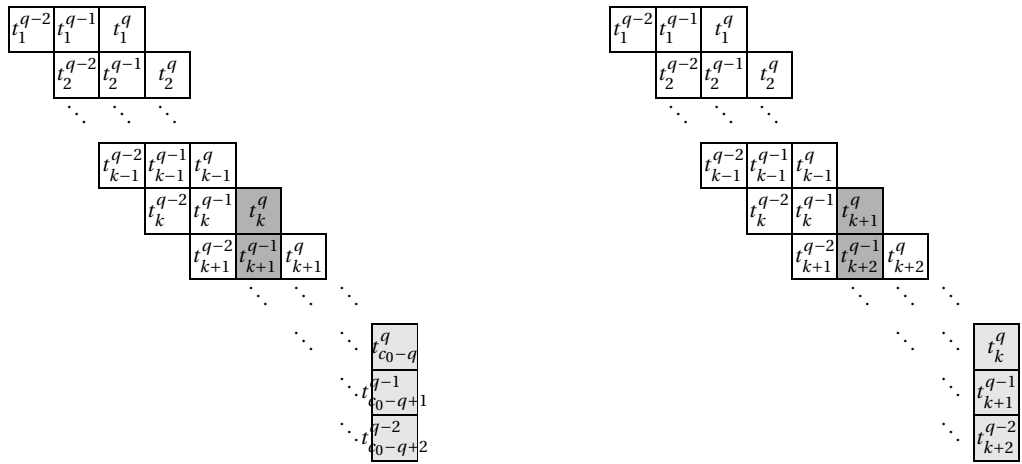
\begin{figure}[h]

 \ytableausetup{boxsize=2.4em}

\begin{tikzpicture}[font=\tiny]

\node at (0,0){\begin{ytableau}
t_1^{q-2}&t_1^{q-1}&t_1^{q}\\
\none&t_2^{q-2}&t_2^{q-1}&t_2^{q}\\
\none&\none[\ddots]&\none[\ddots]&\none[\ddots]\\
\none&\none&t_{k-1}^{q-2}&t_{k-1}^{q-1}&t_{k-1}^{q}\\
\none&\none&\none&t_{k}^{q-2}&t_{k}^{q-1}&*(gray!60)t_{k}^{q}\\
\none&\none&\none&\none&t_{k+1}^{q-2}&*(gray!60)t_{k+1}^{q-1}&t_{k+1}^{q}\\
\none&\none&\none&\none&\none&\none[\ddots]&\none[\ddots]&\none[\ddots]\\
\none&\none&\none&\none&\none&\none&\none[\ddots]&\none[\ddots]&*(gray!20)t^q_{\col-q}\\
\none&\none&\none&\none&\none&\none&\none&\none[\ddots]&*(gray!20)t^{q-1}_{\col-q+1}\\
\none&\none&\none&\none&\none&\none&\none&\none[\ddots]&*(gray!20)t^{q-2}_{\col-q+2}\\
\end{ytableau}};

\node at (8,0){\begin{ytableau}
t_1^{q-2}&t_1^{q-1}&t_1^{q}\\
\none&t_2^{q-2}&t_2^{q-1}&t_2^{q}\\
\none&\none[\ddots]&\none[\ddots]&\none[\ddots]\\
\none&\none&t_{k-1}^{q-2}&t_{k-1}^{q-1}&t_{k-1}^{q}\\
\none&\none&\none&t_{k}^{q-2}&t_{k}^{q-1}&*(gray!60)t_{k+1}^{q}\\
\none&\none&\none&\none&t_{k+1}^{q-2}&*(gray!60)t_{k+2}^{q-1}&t_{k+2}^{q}\\
\none&\none&\none&\none&\none&\none[\ddots]&\none[\ddots]&\none[\ddots]\\
\none&\none&\none&\none&\none&\none&\none[\ddots]&\none[\ddots]&*(gray!20)t^q_{k}\\
\none&\none&\none&\none&\none&\none&\none&\none[\ddots]&*(gray!20)t^{q-1}_{k+1}\\
\none&\none&\none&\none&\none&\none&\none&\none[\ddots]&*(gray!20)t^{q-2}_{k+2}\\
\end{ytableau}};

\end{tikzpicture}
\caption{We denote entry $\TDo_{i,i+q}$ by $t^q_i$. The tableau $\TDo$ is on the left and $\SDo$ is on the right. }
\label{fig:TPrimeSPrime}
\end{figure}

\begin{example}
\label{ex:ForFinalProofOfMultOne}
Let $n=27$, $k=5$, $\gamma=(4,4,4,4,3,2,2,2,2)$, and $\lambda=(4,3,3,3,3,3,3,3,2)$. There are $3$ diagonals ending in $\gamma\setminus\lambda$, which is in column $\col=4$. For any $\cola<\col$ and any diagonal $q$ ending in $\gamma\setminus\lambda$, there are more cells in and below diagonal $q$ in $\cola$ than in and below $q$ in $\col$. See Figure~\ref{fig:exForFinalProofOfMultOne}.

\end{example}

 \ytableausetup{boxsize=1.5em}
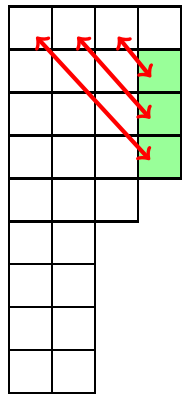
\begin{figure}[h]

\begin{tikzpicture}

\node at (0,0) {
\begin{ytableau} 
 \quad&\quad&\quad&\quad\\
 \quad&\quad&\quad&*(green!40)\quad\\
 \quad&\quad&\quad&*(green!40)\quad\\
 \quad&\quad&\quad&*(green!40)\quad\\
 \quad&\quad&\quad\\
\quad&\quad\\
\quad&\quad\\
\quad&\quad\\
\quad&\quad\\
\end{ytableau}};

\node (A)[red] at (.85,1.5){}; 
\node (B)[red] at (.2,2.3) {};
\draw[ultra thick,red,<->] (A)--(B); 

\node (A1)[red] at (.85,.95){}; 
\node (B1)[red] at (-.35,2.3) {};
\draw[ultra thick,red,<->] (A1)--(B1); 

\node (A2)[red] at (.85,.4){}; 
\node (B2)[red] at (-.9,2.3) {};
\draw[ultra thick,red,<->] (A2)--(B2);

\end{tikzpicture}
\caption{Any non-green cell on an arrow will have more cells beneath it than the green cell on the same arrow. See Example~\ref{ex:ForFinalProofOfMultOne}}
\label{fig:exForFinalProofOfMultOne}
\end{figure}

\begin{proposition}\label{prop:sinuOK}
Let $(\eta,\SIm)$ and $(\mu,\TIm)$ be such that $\eta^+=\mu^+=\standComp$ and
both $\TIm$ and $\SIm$ are elements of $\I$. Assume $(\mu,\TIm)$ is in the near-image. \omitt{NOT in the
image of $\Map$, but for some $i$, $(s_i\mu,T)$ is in the image of
$\Map$.} \omitt{Let $\SDo=\phiInv{\eta}{S}$ and $\TDo=\phiInv{\mu}{T}$.} Then, if $\eta\neq\mu$ or $\TIm\neq \SIm$, then $\wt{\eta}{\SIm}\neq\wt{\mu}{\TIm}$.
\end{proposition}

\begin{proof}
The argument is the same as in the proof of Proposition~\ref{prop:bigPropOne}. 

Although there is a violation in $\TDo$, the entries in each diagonal
are still increasing from north-east to south-west and the diagonals
of $\TDo$, as sets of entries, are the same as the diagonals of
$\SDo$. Since violations in $\SDo$ can only be on the boder to $\gMinusl$,
formula \eqref{eq:diagperms} still holds and we obtain again a
violation in column $\cola$ of $\SDo$, which is not possible by
Claim~\ref{claim:phiInv}.
\end{proof}

\section{Equivariance of $\Map$}
\label{sec:looseEnds}
\omitt{\subsection{TODO??}
\begin{itemize}
\item Row reading tableau\sftodo{Do we use row-reading tabs?}
 \end{itemize}}
\subsection{Introduction}
We begin by recalling the bases $\{v_T\}$ (Jucys-Murphy, special case of the bases in Section~\ref{seminormal}) of the Specht Module and $\{\jack{\mu}{T}\}$ (Section~\ref{sec:genJack}) of the standard modules that we use.  See \cite{Gri2} for more information.


  $$\iop_i v_T=\begin{cases}
  v_{s_iT}&\text{if $\ell(s_iT)>\ell(T)$}\\
  \frac{\left(\CT T^{-1}(i)-\CT T^{-1}(i+1)\right)^2-1}{\left(\CT T^{-1}(i)-\CT T^{-1}(i+1)\right)^2}v_{s_iT}&\text{if $\ell(s_iT)<\ell(T)$}\\
0&\text{if $s_iT$ is not a $\syt$.}
  \end{cases}$$


The map $\Map$, where
\omitt{$\Map:\stanmod(\gamma)\to\stanmod(\lambda)$}$\Map:S^{\gamma}\to\stanmod(\lambda)$, $\Map:\jacko{T}\mapsto
\mapconst{\TDo}\jack{\mu'}{\TIm}$, and $\mapconst{\TDo}$ is a constant
depending on $\TDo$, was defined in Section~\ref{subsec:maps}. 

The point of the current section is to show that $\Mapa{\sigma_i\jacko{\TDo}}=\sigma_i\Mapa{\jacko{\TDo}}$ and that the constants $\{\mapconst{\TDo}\}$ for $T$ a SYT are well-defined.

To avoid all the cases involved in the expression for $\sigma_i$ in \eqref{eq:sigmajack} as much as possible, define $\sigconst{\mu}{T}{i}$ by 
$$\sigma_i\jack{\mu}{T}=\sigconst{\mu}{T}{i}\jack{\tilde{\mu}}{\tilde{T}}.$$ 

$$\sigconst{\mu}{T}{i}=\begin{cases}
1&\text{if $\mu_i>\mu_{i+1}$}\\
\frac{\delta^2-c^2}{\delta^2}&\text{if $\mu_i<\mu_{i+1}$}\\
1&\text{if $\mu_i=\mu_{i+1}$, $\ell(s_{j-1}T)>\ell(T)$, and}\\
&\text{$s_{j-1}T$ is a standard Young tableau}\\
1-\left(\frac{1}{\ct{T}{j}-\ct{T}{j-1}}\right)^2&\text{if $\mu_i=\mu_{i+1}$, $\ell(s_{j-1}T)\ell(T)$, and}\\
&\text{$s_{j-1}T$ is a standard Young tableau}\\
0&\text{$s_{j-1}$ is not a standard Young tableau}
\end{cases}$$

If $s_{j-1}T$ is a standard Young tableau and $\ell(s_{j-1}T)>\ell(T)$, then
$$\Mapa{\sigma_i\jacko{T}}=\Mapa{\jacko{s_{j-1}T}}=\mapconst{s_{j-1}T}\jack{\eta}{(s_{j-1}T)'}$$
and $$\sigma_i\Mapa{\jacko{T}}=\mapconst{T}\sigma_i\jack{\mu}{\TIm}=\sigconst{\mu}{\TIm}{i}\mapconst{T}\jack{\tilde{\mu}}{\widetilde{\TIm}}.$$ Here $j=w_{\zeroComp}(i)=n-i+1$.

Then if $\ell(s_{j-1}T)>\ell(T)$, $j=w_\zeroComp(i)$, and $s_{j-1}T$ is a standard Young tableau, we have

  $$b_{s_{j-1}T}=\frac{\sigconst{\mu}{\TIm}{i}}{\sigconst{\zeroComp}{T}{i}}\mapconst{T}=\sigconst{\mu}{\TIm}{i}\mapconst{T},$$ since $\ell(s_{j-1}T)>\ell(T)$ and where $\mu$ is defined by $\Mapa{\jacko{T}}=\mapconst{T}\jack{\mu}{\TIm}$.

We claim
\begin{enumerate}
\item\label{sameimage} $\tilde{\mu}=\eta$ and $(s_{j-1}T)'=\widetilde{\TIm}$,
\item\label{defined} if $w T_0$ is a standard Young tableau and $s_{i_1}\cdots s_{i_k}$ is a reduced word for $w$, then for all $j$, $1\leq j\leq k$
$$s_{i_j}\cdots s_{i_k}T_0$$ is a standard Young tableau,
\item\label{braid} $b_{s_{j-1}s_{j-2}s_{j-1}T}=b_{s_{j-2}s_{j-1}s_{j-2}T}$, and 
  \item\label{commute} $b_{s_is_jT}=b_{s_js_iT}$ where $|i-j|>1$.
\end{enumerate}

We discuss \eqref{braid} and \eqref{commute} only. 

\subsection{Discussion of \eqref{braid}}
\begin{claim}
\label{cl:simu}
  Suppose $\Mapa{\jacko{T}}=\mapconst{T}\jack{\mu}{\TIm}$ and
  $\Mapa{\jacko{s_{k-1}T}}=\mapconst{s_{k-1}T}\jack{\mu'}{(s_{k-1}T)'}$,
  where $s_{k-1}T$ is a standard Young tableau and
  $k=w_{\zeroComp}(h)$. We have $\mu'=s_h\mu$
\end{claim}
\begin{proof} If neither $k$ nor $k-1$ are entries cells in $\gMinusl$ in $T$, then $\mu_h=\mu_{h+1}=0$ and $\mu'_h=\mu'_{h+1}=0$. Since $s_{k-1}T$ is a standard Young tableau, not both $k$ and $k-1$ are entries in $\gMinusl$. If $k-1$ is in $\gMinusl$ and $k$ is not, then $\mu_{h+1}=m$ and $\mu_h=0$, and $\mu'_{h+1}=0$ and $\mu_h=m$. The situation is similar for $k$ in $\gMinusl$ and $k-1$ not.  
\end{proof}

We need to show
\begin{equation}
  \label{eq:braid}
  b_{s_{j-1}s_{j-2}s_{j-1}T}=b_{s_{j-2}s_{j-1}s_{j-2}T}.
  \end{equation}

Both the left and the right hand side of \eqref{eq:braid} are the product of $\sigconst{\cdot}{\cdot}{\cdot}$'s times $\mapconst{T}$. Equation \eqref{eq:braid} is equivalent to
\begin{multline}
    \overbrace{\sigconst{s_{i+1}s_i\mu}{(s_{j-2}s_{j-1}T)'}{i}}^{A}
    \overbrace{\sigconst{s_i\mu}{(s_{j-1}T)'}{i+1}}^{B}
    \overbrace{\sigconst{\mu}{\TIm}{i}}^{C}\\
    \quad\quad=\overbrace{\sigconst{s_{i}s_{i+1}i\mu}{(s_{j-1}s_{j-2}T)'}{i+1}}^{D}\overbrace{\sigconst{s_{i+1}\mu}{(s_{j-2}T)'}{i}}^{E}\overbrace{\sigconst{\mu}{\TIm}{i+1}}^{F}
\end{multline}

We claim that $A=F$, $B=E$, and $C=D$. We show $A=F$, which is
\eqref{eq:AF} below. The other cases are similar.

\begin{equation}
  \label{eq:AF}
  \sigconst{s_{i+1}s_i\mu}{(s_{j-2}s_{j-1}T)'}{i}=\sigconst{\mu}{\TIm}{i+1}.
  \end{equation}

The assumptions here are
\begin{enumerate}
\item $\ell(s_{j-2}s_{j-1}s_{j-2}T)=\ell(T)+3$. 
\item $j,j-1,j-2$ are all in different rows and columns in $T$. Otherwise, either $s_{j-1}T$ or $s_{j-2}T$ is not a standard Young tableau. Them means at most one of $\{j,j-1,j-2\}$ is in $\gMinusl$.
\item Of the three $\{j,j-1,j-2\}$, $j$ is the lowest row, $j-2$ in the highest, and $j-1$ between. Otherwises, we'd have $\ell(s_{j-1}T)<\ell(T)$ or $\ell(s_{j-2}T)<\ell(T)$.
 \item
   $s_{i+1}s_i\mu=(\mu_1,\ldots,\mu_{i-1}\mu_{i+1}\mu_{i+2}\mu_i\mu_{i+3},\dots,\mu_n)$,
   so that $(s_{i+1}s_i\mu)_i=\mu_{i+1}$ and
   $(s_{i+1}s_i\mu)_{i+1}=\mu_{i+2}$
  \end{enumerate}
By the definition of $\sigconst{\cdot}{\cdot}{\cdot}$, to prove
\eqref{eq:AF}, it is enough to show that
\begin{equation}
  \label{eq:ctone}\ct{\TIm}{w_\mu(i+1)}=\ct{(s_{j-2}s_{j-1}T)'}{w_{s_{i+1}s_i\mu}(i)}
  \end{equation}
and
\begin{equation}
  \label{eq:cttwo}\ct{\TIm}{w_\mu(i+2)}=\ct{(s_{j-2}s_{j-1}T)'}{w_{s_{i+1}s_i\mu}(i+1)}
  \end{equation}
and that we are in the same case when determining which expression to use for $\sigconst{\mu}{\TIm}{i+1}$ and $\sigconst{s_{i+1}s_i\mu,}{(s_{j-2}s_{j-1}T)'}{i}.$

Table~\ref{braidpositiontable} tracks the entries in three cells in $\gamma$ and $\lambda$. Each row corresponds to a cell. The table shows that \eqref{eq:ctone} and \eqref{eq:cttwo} hold.

\begin{center}
  \begin{tabular}{C|C|C||C|C|C}
    \text{Entry in}&\text{Entry in}&\text{Entry in}&\text{Entry in}&\text{Entry in}&\text{Entry in}\\
    T&\rev{T}&\TIm&s_{j-2}s_{j-1}T&\rev{(s_{j-2}s_{j-1}T)}&(s_{j-2}s_{j-1}T)'\\
    \hline
    j&i&w_\mu(i)&j-2&i+2&w_{s_{i+1}s_i\mu}(i+2)\\
    j-1&i+1&w_\mu(i+1)&j&i&w_{s_{i+1}s_i\mu}(i)\\
       j-2&i+2&w_\mu(i+2)&j-1&i+1&w_{s_{i+1}s_i\mu}(i+1)\\ 
  \end{tabular}
  \captionof{table}{Cell entries under $\Map$}\label{braidpositiontable}
  \end{center}

Since $(s_{i+1}s_i\mu)_i=\mu_{i+1}$ and
$(s_{i+1}s_i\mu)_{i+1}=\mu_{i+2}$, I only need to show that if
$\mu_{i+1}=\mu_{i+1}$, then $$\ell(s_{j_1-1}\TIm)>\ell(\TIm)\text{ iff
}\ell(s_{j_2-1}(s_{j-2}s_{j-1}T)')>\ell((s_{j-2}s_{j-1}T)'),$$ where
$j_1=w_{mu}(i+1)$ and $j_2=w_{s_{j-2}s_{j-1}\mu}(i)$. But then the
positions of $j_1$ in $\TIm$ and $j_2$ in $(s_{j-2}s_{j-1}T)'$ are the
same, and since $\mu_{i+1}=\mu_{i+2}$, so are the positions of $j_1-1$
and $j_2-1$. Thus we are in the same case when determining which expression to use for $\sigconst{\mu}{\TIm}{i+1}$ and $\sigconst{s_{i+1}s_i\mu,}{(s_{j-2}s_{j-1}T)'}{i}.$
\omitt{
\subsection{Discussion of \eqref{defined}}
Define the inversion set of a tableau $T$ to be
$$\inv T=\{(i,j)\mid 1\leq i<j\leq n, j\text{ is in a higher row than }i\}.$$ Note that a tableau $T$ is equal to $w T_0$ for a unique $w\in\Sym_n$ and that $\inv T$ agrees with $\inv w$.

\begin{claim}\label{cl:invset}
  If $w=s_{i_1}\ldots s_{i_k}$ and $T=w T_0$, then
  $$\inv T=\{(s_{i_1}s_{i_2}\ldots s_{i_{j-1}}(i_j),s_{i_1}s_{i_2}\ldots s_{i_{j-1}}(i_j+1))\mid 1\leq i\leq k\}.$$
\end{claim}

  No proof of claim. Find some standard ref. Be a little careful about left vs. right inversions.

  Now suppose $T=w T_0$ and $w=s_{i_1}w'$. I want to use induction on $\ell(w)=\ell(T)$, so it is enough to show $w' T_0$ is a standard Young tableau.

  $w' T_0=s_{i_1}T$, which is not a standard Young tableau if and only
  if $i_1$ and $i_1+1$ are in the same row or column. But
  $(i_1,i_1+1)\in\inv T$ by Claim~\ref{cl:invset}, which means they
  cannot be.
} 
\omitt{
  \subsection{Discussion of \eqref{sameimage}}

Reminder: if $s_{j-}T$ is a standard Young tableau and
$\ell(s_{j-1}T)>\ell(T)$, then
$$\Mapa{\sigma_i\jacko{T}}=\Mapa{\jacko{s_{j-1}T}}=\mapconst{s_{j-1}T}\jack{\eta}{(s_{j-1}T)'}$$
and $$\sigma_i\Mapa{\jacko{T}}=\mapconst{T}\sigma_i\jack{\mu}{\TIm}=\sigconst{\mu}{\TIm}{i}\mapconst{T}\jack{\tilde{\mu}}{\widetilde{\TIm}}.$$ Here $j=w_{\zeroComp}(i)=n-i+1$.

By Claim~\ref{cl:simu}, $\eta=s_i\mu$. If $\mu_i\neq\mu_{i+1}$, then $\tilde{\mu}=s_i\mu=\eta$. If $\mu_i=\mu_{i+1}$, then $s_i\mu=\mu=\tilde{\mu}$.

\begin{center}
\begin{tabular}{C|C|C|C||C|C|C|C:C}
\text{Entry in}&\text{Entry in}&\text{Entry in}&\text{Entry
  in}&\text{Entry in}&\text{Entry in}&\text{Entry in}&\text{Entry
  in}&\text{Entry
  in}\\ T&s_{j-1}T&(s_{j-1}T)'&\rev{(s_{j-1}T)}&T&\rev{T}&\TIm&\widetilde{\TIm}&\widetilde{\TIm}\\ 
&&&&&&&\text{(if
  $\mu_i\neq\mu_{i+1}$)}&\text{(if $\mu_i=\mu_{i+1}$)}\\ \hline
j&j-1&i+1&w_\eta(i+1)&j&i&w_\mu(i)&w_\mu(i)&w_\mu(i+1)\\
j-1&j&i&w_\eta(i)&j-1&i+1&w_\mu(i+1)&w_\mu(i+1)&w_\mu(i)\\

\end{tabular}
\captionof{table}{Cell entries of $\TDo$ under $\Map$}
\end{center}
For the last column, let $j_1=w_\mu(i)$ so that $\widetilde{\TIm}=s_{j_1-1}\TIm$. If $\mu_i=\mu_{i+1}$, then $w_\mu(i+1)=j_1-1$
}
  \subsection{Discussion of \eqref{commute}}
  Suppose $\mid j_1-j_2\mid>1$. I want to show that $\mapconst{s_{j_1-1}s_{j_2-1}T}=\mapconst{s_{j_2-1}s_{j_1-1}T}$.

  Let $i_1=w_0(j_1)$ and $i_2=w_0(j_2)$. Let $\mu^1$ be defined by $\Mapa{\jacko{s_{j_1-1}T}}=\mapconst{s_{j_1-1}T}\jack{\mu^1}{(s_{j_1-1}T)'}$, $\mu^2$ by $\Mapa{\jacko{s_{j_2-1}T}}=\mapconst{s_{j_2-1}T}\jack{\mu^2}{(s_{j_2-1}T)'}$, and $\mu$ by $\Mapa{\jacko{T}}=\mapconst{s_T}\jack{\mu}{\TIm}$. Then

  \begin{eqnarray*}
    \mapconst{s_{j_1-1}s_{j_2-1}T}&=&\sigconst{\mu^2}{(s_{j_2-1}T)'}{i_1}\mapconst{s_{j_2-1}T}\\
    &=&\sigconst{\mu^2}{(s_{j_2-1}T)'}{i_1}\sigconst{\mu}{\TIm}{i_2}\mapconst{T}
  \end{eqnarray*}

  and

  \begin{eqnarray*}
    \mapconst{s_{j_2-1}s_{j_1-1}T}&=&\sigconst{\mu^1}{(s_{j_1-1}T)'}{i_2}\mapconst{s_{j_1-1}T}\\
    &=&\sigconst{\mu^1}{(s_{j_1-1}T)'}{i_2}\sigconst{\mu}{\TIm}{i_1}\mapconst{T}
  \end{eqnarray*}

  We need to show
  $$\overbrace{\sigconst{\mu^2}{(s_{j_2-1}T)'}{i_1}}^A\overbrace{\sigconst{\mu}{\TIm}{i_2}}^B=\overbrace{\sigconst{\mu^1}{(s_{j_1-1}T)'}{i_2}}^C\overbrace{\sigconst{\mu}{\TIm}{i_1}}^D.$$

  We claim $A=D$ and $B=C$, and will show $A=D$:
  \begin{equation}
    \sigconst{\mu^2}{(s_{j_2-1}T)'}{i_1}=\sigconst{\mu}{\TIm}{i_1}
  \end{equation}
\begin{center}
  \begin{tabular}{C|C|C||C|C|C}
    \text{Entry in}&\text{Entry in}&\text{Entry in}&\text{Entry in}&\text{Entry in}&\text{Entry in}\\
    \TDo&\rev{\TDo}&\TIm&s_{j_2-1}T&\rev{(s_{j_2-1}T)}&(s_{j_2-1}T)'\\
    \hline
    j_1-1&i_1+1&w_\mu(i_1+1)&j_1-1&i_1+1&w_{\mu^2}(i_1+1)\\
    j_1&i_1&w_\mu(i_1)&j_1&i_1&w_{\mu^2}(i_1)\\
    j_2-1&i_2+1&w_\mu(i_2+1)&j_2&i_2&w_{\mu^2}(i_2)\\
    j_2&i_2&w_\mu(i_2)&j_2-1&i_2+1&w_{\mu^2}(i_2+1)\\
        \end{tabular}

\omitt{ 
  \begin{tabular}{C|C|C||C|C|C|C}
    \text{Entry in}&\text{Entry in}&\text{Entry in}&\text{Entry in}&\text{Entry in}&\text{Entry in}&\text{Entry in}\\
    T&\rev{T}&\TIm&T&s_{j_2-1}T&\rev{(s_{j_2-1}T)}&(s_{j_2-1}T)'\\
    \hline
    j_1-1&i_1+1&w_\mu(i_1+1)&j_1-1&j_1-1&i_1+1&w_{\mu^2}(i_1+1)\\
    j_1&i_1&w_\mu(i_1)&j_1&j_1&i_1&w_{\mu^2}(i_1)\\
    j_2-1&i_2+1&w_\mu(i_2+1)&j_2-1&j_2&i_2&w_{\mu^2}(i_2)\\
    j_2&i_2&w_\mu(i_2)&j_2&j_2-1&i_2+1&w_{\mu^2}(i_2+1)\\
        \end{tabular}
 } 
\captionof{table}{Cell entries of $\TDo$ under $\Map$}
\end{center}

  By Claim~\ref{cl:simu}, $w_{\mu^2}=w_{s_{i_2}\mu}$. This means that
  if $\mu_{i_2}=\mu_{i_2+1}$, then $w_{\mu^2}=w_{\mu}$ and if
    $\mu_{i_2}\neq\mu_{i_2+1}$, then
      \begin{equation}\label{eq:mmu2}w_{\mu^2}(k)=\begin{cases}w_{\mu}(k)&\text{if $k\neq i_2,i_2+1$}\\
      w_{\mu}(i_2+1)&\text{if $k=i_2$}\\
      w_{\mu}(i_2)&\text{if $k= i_2+1$}\\
      \end{cases}\end{equation}

  Thus,
  \begin{equation}
    \label{eq:sj2-1T}
    (s_{j_2-1}T)'=\begin{cases}\TIm&\text{ if $\mu_{i_2}\neq\mu_{i_2+1}$}\\
  s_{h-1}\TIm&\text{ if $\mu_{i_2}=\mu_{i_2+1}$ and $h=w_\mu(i_2)$}\\
    \end{cases}
    \end{equation}

  As is Section~\ref{braid}, we need to show
  \begin{enumerate}
  \item \label{comKcase}Show we are in the same $\sigconst{\cdot}{\cdot}{\cdot}$ case.
    \item\label{comKval} Show that we have the same $\sigconst{\cdot}{\cdot}{\cdot}$ value. 
  \end{enumerate}
  \eqref{comKcase}: Since $\mu^2=s_{i_2}\mu$ and $\mid j_2-j_1\mid >1$, we have
\begin{equation}\label{eq:mu2}
  \mu^2_{i_1}=\mu_{i_1}\text{ and }\mu^2_{i_1+1}=\mu_{i_1+1}.
\end{equation}
If $\mu_{i_1}\neq\mu_{i_1+1}$, then it is immediate that we are in the same $\sigconst{\cdot}{\cdot}{\cdot}$-case. If $\mu_{i_1}=\mu_{i_1+1}$, then we must show that
\begin{equation}\label{eq:comKcase}\ell(s_{k-1}\TIm)>\ell(\TIm)\Leftrightarrow \ell(s_{k-1}(s_{j_2-1})')>\ell((s_{j_2-1}T)'),\end{equation} where $k=w_{\mu^2}(i_1)=w_\mu(i_1)$. Since $(s_{j_2-1}T)'$ and $\TIm$ differ at most in the cells containing $w_{\mu}(i_2)$ and $w_\mu(i_2+1)$ and only when $\mu_{i_2}=\mu_{i_2+1}$, and $\{w_\mu(i_1),w_\mu(i_1+1)\}$ is disjoint from $\{w_\mu(i_2),w_\mu(i_2+1)\}$, \eqref{eq:comKcase} is true.

 \eqref{comKval}: Since $\mu^2=s_{i_2}\mu$ and $\mid i_1-i_2\mid>1$, we have $\mu_{i_1}^2=\mu_{i_1}$, $\mu_{i_1+1}^2=\mu_{i_1+1}$, $w_\mu(i_1)=w_{\mu^2}(i_1)$, and $w_\mu(i_1+1)=w_{\mu^2}(i_1+1)$. What's more, $\TIm=(s_{j_2-1}T)'$ or $s_{w_\mu(i_2)-1}\TIm=(s_{j_2-1}T)'$. All together, we have $\ct{\TIm}{w_\mu(i_1)}=\ct{(s_{j_2-1}T)'}{w_\mu^2(i_1)}$ and $\ct{\TIm}{w_\mu(i_1+1)}=\ct{(s_{j_2-1}T)'}{w_\mu^2(i_1+1)}$, which shows that no matter what case we are in, $\sigconst{\mu^2}{(s_{j_2-1}T)'}{i_1}=\sigconst{\mu}{\TIm}{i_1}$

\bibliographystyle{amsalpha}

\end{document}